\documentclass[10pt]{article}
\usepackage[utf8]{inputenc} 
\usepackage[fleqn]{amsmath}

\usepackage{amsthm}
\usepackage{amssymb}

\usepackage[T1]{fontenc}
\usepackage{lmodern}

\usepackage[dvipsnames]{xcolor}
\usepackage[a4paper]{geometry}
\usepackage[english]{babel}
\usepackage{extsizes}
\usepackage{xcolor}
\usepackage{fancybox}
\usepackage{graphicx}
\usepackage{sectsty}
\usepackage{enumerate}
\usepackage{multicol}
\usepackage[french]{varioref}
\usepackage{mathtools}
\usepackage{mathrsfs}
\usepackage{blkarray}
\usepackage{framed}
\usepackage{multirow}
\usepackage{pgfplots}
\usepackage[export]{adjustbox}
\usepackage{minibox}

\newcommand{\hide}[1]{}

\usepackage{mathbbol}
\usepackage{hyperref}%importante l'ordine hyperref --> amsthm&(geometry-->todonotes) --> cleveref
\hypersetup{
	colorlinks,
	linkcolor={red!50!black},
	citecolor={blue!50!black},
	urlcolor={black}
}

\usepackage{tkz-euclide}
\usetkzobj{all}

\newtheorem{Th}{Theorem}
\newtheorem{Prop}[Th]{Proposition}
\newtheorem{Le}[Th]{Lemma}
\newtheorem{Cor}[Th]{Corollary}
\theoremstyle{definition}

\newtheorem*{Rk}{Remark}

\newcommand{\dint}{\mathrm{d}}
\newcommand{\vol}[2]{ \mathrm{V}_{#1} \left( #2 \right) }

\newcommand{\window}{\mathbf{W}_\rho}

\newcommand{\neighbours}{ \mathcal{N} }

\newcommand{\RR}{\mathbf{R}}

\newcommand{\NN}{\mathbf{N}}
\newcommand{\XX}{\eta}
\newcommand{\SSS}{\mathbf{S}^{d-1}}
\newcommand{\del}{\mathrm{Del}}
\newcommand{\typ}{\mathcal{D}^{0}}

\newcommand{\conv}[2][n]{\underset{#1\rightarrow #2}{\longrightarrow}}
\newcommand{\eq}[2][n]{\underset{#1\rightarrow #2}{\sim}}

\newcommand{\EEE}[1]{\operatorname{\mathbb{E}}\left[\,#1\,\right]}

\newcommand{\PPP}[1]{\operatorname{\mathbb{P}}\left(\,#1\,\right)}
\newcommand{\PP}{\operatorname{\mathbb{P}}}
\newcommand{\ind}[1]{\mathbb{1}_{\{#1\}}\,}

\renewcommand{\mid}{ : }

\begin{document}
  \pagenumbering{arabic}
  \pagestyle{plain}
  
\author{Gilles BONNET\footnote{Ruhr-Universit\"{a}t Bochum, Universit\"{a}tstr, 150, Fakult\"{a}t f\"{u}r Mathematik, 44780 Bochum, Germany. E-mail: gilles.bonnet@ruhr-uni-bochum.de}\and
Nicolas CHENAVIER\footnote{Universit\'e du Littoral C\^ote d'Opale, EA 2797, LMPA, 50 rue Ferdinand Buisson, F-62228 Calais, France. E-mail: nicolas.chenavier@univ-littoral.fr}}

  \title{The maximal degree in a Poisson-Delaunay graph}
  \maketitle

\begin{abstract}
We investigate the maximal degree in a Poisson-Delaunay graph in $\RR^d$, $d\geq 2$, over all nodes in the window $\window:= \rho^{1/d}[0,1]^d$ as $\rho$ goes to infinity. The exact order of this maximum is provided in any dimension. In the particular setting $d=2$, we show that this quantity  is concentrated on two consecutive integers with high probability. An extension of this result is discussed when $d\geq 3$.
\end{abstract}


%\strut


  \section{Introduction}

 Delaunay graphs are a very popular structure in computational geometry~\cite{aurenhammer2013voronoi} and are extensively used in many areas such as surface reconstruction, mesh generation, molecular modeling, and medical image segmentation, see e.g. \cite{cazals2006delaunay,cheng2012delaunay}. The book by Okabe et al. \cite{OBSC} gives a taste of the richness of the theory of these graphs and of the variety of their applications.  In this paper, we consider a Poisson-Delaunay graph that is a random Delaunay graph based on a stationary Poisson point process in $\RR^d$, $d\geq 2$. 
 
Recently, extremes of various quantities associated with Poisson-Delaunay graphs have been investigated by Chenavier, Devillers and Robert. In \cite{ChenDev} the length of the shortest path between two distant vertices is considered.  In \cite{Chen,ChenRob}, the extremes studied are the largest or smallest values of a given geometric characteristic, such as the volume or the circumradius, over all simplices in the Poisson-Delaunay graph with incenter in a large window. For a broad panorama of extreme values arising from construction based on a Poisson point process, we refer the reader to \cite{ST2}. 
 
 However, all the  distributions of the random variables which are considered in the literature have a probability density function.  In this paper, we deal with the case of a discrete random variable, namely the maximal degree. More precisely, let $\XX$ be a stationary Poisson point process in $\RR^d$. Without restriction, we assume that the intensity of $\XX$ equals 1. Let  $\window=\rho^{1/d}[0,1]^d$, where $\rho$ is a positive real number. We investigate the asymptotic behaviour, as $\rho$ goes to infinity, of the following random variable:
\[\Delta_\rho:=\max_{x\in \XX\cap \window} d_\XX(x), \] where $d_\XX(x)$ denotes the degree of any node $x\in \XX$ in the Poisson-Delaunay graph associated with $\XX$, i.e.\ the number of (non-oriented) edges passing through $x$ (see Figure \ref{fig:del}). The maximal degree of random combinatorial graphs has been extensively investigated, see e.g.\ \cite{Bo,CGS,DGNPS,GW,GMN,MR,MSW2}. Much less has been done when the vertices are given by a point process and the edges built according to geometric constraints. To the best of the author's knowledge, one of the first results on the maximal degree in a Poisson-Delaunay graph was due to Bern \textit{et al.} (see Theorem 7 in \cite{BEY}) who showed that \begin{equation}
\label{eq:bose}\EEE{\Delta_\rho}=\Theta\left(\frac{\log \rho}{\log\log \rho} \right)
\end{equation} in any dimension $d\geq 2$. Broutin \textit{et al.} \cite{BDH} went on to provide a new bound for $\Delta_\rho$ in the following sense: when $d=2$, with probability tending to 1, the maximal degree $\Delta_\rho$ is less than $(\log \rho)^{2+\xi}$, with $\xi>0$. Our main theorem significantly improves  these two results in dimension two.

  \begin{Th}
    \label{Th:maxdegree}
    Let $\Delta_\rho$ be the maximal degree in  a planar Poisson-Delaunay graph over all nodes in $\window=\rho^{1/2}[0,1]^2$. Then there exists a deterministic function $\rho\mapsto I_\rho$, $\rho>0$, with values in $\NN$, such that
    \begin{enumerate}[(i)]
      \item  $\PPP{\Delta_\rho\in\{I_\rho,I_\rho+1\}}\conv[\rho]{\infty}1$;
      \item $I_\rho\eq[\rho]{\infty}\frac{1}{2}\cdot \frac{\log \rho}{\log\log \rho}$. 
    \end{enumerate}
  \end{Th}
In particular, our result provides the exact order of the maximal degree  and claims that, with high probability, the maximal degree is concentrated on two consecutive values. 
As observed in Figure \ref{fig:Maxdegree}, the concentration is already visible for $\rho =10^6$. On the other hand, the estimate of $I_\rho$ is good only for much larger values of $\rho$ because of the extremely slow growth of the logarithm. This will be discussed further at the end of Section \ref{sec:typicaldegree}.

\begin{figure}
\centering
\begin{minipage}[b]{.41\textwidth}
  \centering
    \includegraphics[scale=0.45]{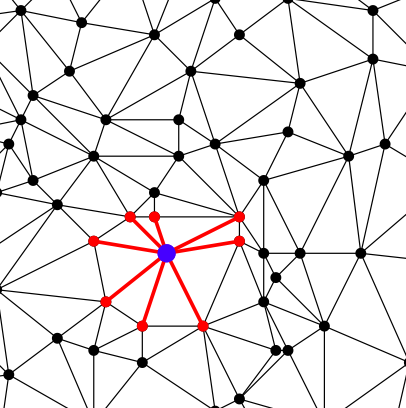}
%   \adjincludegraphics[width=\linewidth,trim={{.18\width} {.31\height} {.47\width} {.14\height}},clip=true]{figures/figdel2.pdf}
    \caption{\label{fig:del}A node (blue) with its neighbors (red) which maximizes the degree in a planar Delaunay graph observed in a window.}
    \vspace{1Em}
\end{minipage}%
\hfill%
\begin{minipage}[b]{.51\textwidth}
  \centering
  \pgfplotstableread[row sep=\\,col sep=&]{degree & proba \\    13 & 0 \\   14 & 0.02  \\  15 & 0.655 \\   16 & 0.29 \\  17 & 0.03 \\ 18 & 0.002 \\  19 & 0 \\  }\mydata
    \begin{tikzpicture}[scale=0.95]    
        \begin{axis}[
            axis lines=left,
            ybar,
            bar width=.5cm,
            symbolic x coords={13,14,15,16,17,18,19},
            xtick=data,
            ymin=0,
            ymax=1,
            ylabel={Empirical Probability},
            xlabel={Maximal Degree} ]
            \addplot table[x=degree,y=proba]{\mydata}; 
        \end{axis}
    \end{tikzpicture}
    \caption{\label{fig:Maxdegree}Empirical distribution of $\Delta_\rho$,  based on 75000 simulations, of the maximal degree in a planar Poisson-Delaunay graph observed in the window $\mathbf{W}_{10^6}=10^3[0,1]^2$.}
\end{minipage}
\end{figure}

 Our theorem is rather classical in the sense that similar results have already been established in the context of  random combinatorial graphs \cite{Bo,CGS,DGNPS,GW,GMN,MR,MSW2} and Gilbert graph \cite[Th 6.6]{Pr}.   Besides, Anderson \cite{And} proved that the maximum of $n$ independent and identically distributed random variables  is concentrated, with high probability as $n$ goes to infinity,  on two consecutive integers for a wide class of discrete random variables. Kimber \cite{Kim} provided rates of convergence in the particular case  where the random variables are Poisson distributed.  However, two difficulties are added in the context of Poisson-Delaunay graphs. The first one is that the distribution of the typical degree cannot be made explicit. The second one, which constitutes the main difficulty, comes from the dependence between the degrees of the nodes and the geometric constraints in the Poisson-Delaunay graph. 
 
  As a corollary of Theorem \ref{Th:maxdegree}, we can find arbitrary large windows for which the maximal degree is concentrated on only one integer with high probability. 
 
 \begin{Cor}
 \label{cor:onevalue}
Let $\Delta_\rho$ be the maximal degree in  a planar Poisson-Delaunay graph over all nodes in $\window=\rho^{1/2}[0,1]^2$. Then there exists an increasing sequence $(\rho_i)$ converging to infinity  such that \[\PPP{\Delta_{\rho_i} = i}\conv[i]{\infty}1.\]
 \end{Cor}
 
 A weaker version of Theorem \ref{Th:maxdegree}, which deals with the general case $d\geq 2$, is stated below. 
 
 \begin{Th}
 \label{Th:maxdegreed}
 Let $\Delta_\rho$ be the maximal degree in  a Poisson-Delaunay graph over all nodes in $\window=\rho^{1/d}[0,1]^d$, $d\geq 2$. Then there exists a deterministic function $\rho\mapsto J_\rho$, $\rho>0$, with values in $\NN$, such that
    \begin{enumerate}[(i)]
      \item  $\PPP{\Delta_\rho\in\{J_\rho, J_\rho+1, \ldots, J_\rho+l_d\}}\conv[\rho]{\infty}1$, where $l_d = \left\lfloor\frac{d+3}{2} \right\rfloor$;
      \item $J_\rho\eq[\rho]{\infty}\frac{d-1}{2}\cdot \frac{\log \rho}{\log\log \rho}$. 
    \end{enumerate}
 \end{Th}
 In particular, when $d=2$, the above result claims that the maximal degree is concentrated on three consecutive values, which is less accurate than Theorem \ref{Th:maxdegree}. When  $d=3$ and $d=4$, this also shows that the maximal degree is concentrated on four consecutive values.   

  Although Theorem \ref{Th:maxdegree} only deals with the two dimensional case,  its proof is significantly more difficult than the one of Theorem  \ref{Th:maxdegreed}.  Indeed, as opposed to Theorem \ref{Th:maxdegree}, we think that Theorem  \ref{Th:maxdegreed} is not optimal in the sense that the maximal degree should also be concentrated on two consecutive integers, and not only on $l_d+1$ integers. The proof of Theorem \ref{Th:maxdegree} extensively uses the fact that the graph is planar. In particular, as an intermediate result to derive Theorem \ref{Th:maxdegree}, we prove that there is no family of five nodes in the Poisson-Delaunay graph which are close to each others and such that their degrees simultaneously exceed $I_\rho$ with high probability. Such a result is essential in our proof and is specific to the two dimensional case.
  
  As a consequence of Theorem \ref{Th:maxdegreed}, the following corollary improves the estimate \eqref{eq:bose}.
 
 \begin{Cor}
 \label{cor:eq}
 Let $\Delta_\rho$ be the maximal degree in a Poisson-Delaunay graph over all nodes in $\window=\rho^{1/d}[0,1]^d$, $d\geq 2$.  Then  $\EEE{\Delta_\rho }\eq[\rho]{\infty}\frac{d-1}{2}\cdot \frac{\log \rho}{\log\log \rho}$.
 \end{Cor}

The paper is organized as follows. In Section \ref{sec:preliminaries}, we give several preliminaries by introducing some notation and by recalling a few known results. In Section \ref{sec:intermediateresults}, we present technical lemmas which will be used to derive Theorems \ref{Th:maxdegree} and \ref{Th:maxdegreed}. In Section \ref{sec:maintheorem}, we prove our main theorems and their corollaries. The proofs of the technical lemmas are given in Section \ref{sec:lemmas}.

    \section{Preliminaries}
  \label{sec:preliminaries}
  
\subsection{Notation}
\label{sec:notation}
We summarize here the notation used throughout the text.

\paragraph{General notation}
We denote by $\NN = \{ 1 , 2 ,\ldots \} $ and $\RR_+ = [0,\infty) $ the sets of positive integers and non-negative numbers, respectively. The $d$-dimensional Euclidean space $ \RR^d $ is endowed with the Euclidean norm $\|\cdot\|$ and with its $d$-dimensional Lebesgue measure $\vol{d}{\cdot}$. We denote by $\mathcal{B}_{+}^{d}$ the set of Borel sets $B\subset\RR^d$ such that $0<\vol{d}{B}<\infty$. The unit sphere with dimension $d-1$ is denoted by $\SSS$.

Now, let $k\in\NN$ be fixed. We use the short notation $x_{1:k} = (x_1,\ldots,x_k) \in (\RR^d)^{k}$, and for such a $k$-tuple of points we write $ s x_{1:k} + t = (s x_1 + t , \ldots , s x_k + t )$ for any $ s \in \RR$ and $ t \in \RR^d$.
We also consider concatenation of such vectors, for example we write $(x_{1:k},y_{1:l}) = (x_1,\ldots,x_k,y_1,\ldots,y_l)$.
For any set $ S $, we denote by $ S^k_{\neq} $ the family of vectors $ (s_1,\ldots,s_k) \in S^k$ such that $s_i \neq s_j$ for any $i\neq j$. If $\chi$ is a finite set, we also denote by $\#\chi$ its cardinality
%

%All integrals in this manuscript are of the form $\int_A f(p) \, \mathrm{d} p $, where $ A $ is a Borel set of $ ( \RR^d )^k $, $d,k\geq 1$ and where the corresponding Lebesgue measure is used.

Given two functions $f,g: \RR\rightarrow\RR$, we write $g(x)\eq[x]{\infty}f(x)$ if and only if $f$ and $g$ are asymptotically equivalent, i.e.\ $\frac{g(x)}{f(x)}\conv[x]{\infty}1$. Moreover, we write $g(x)=O(f(x))$  if and only if there exists a positive number $M$ and a real number $x_0$ such that $|g(x)|\leq M|f(x)|$ for any $x\geq x_0$. When $\frac{g(x)}{f(x)}\conv[x]{\infty}0$ we write $g(x)=o(f(x))$. 

The quantity $c$ denotes a generic constant which depends only on the dimension $d$. We occasionally index the constants when the distinction between several of them need to be made explicit, e.g. when two or more constants appear in a single equation.

\paragraph{Delaunay graph}
We recall that a (undirected) graph $G=(V,E)$ is a set $V$ of vertices together with a set $E$ of edges with no orientation. 
Given a graph $G$, we denote the set of neighbors of a vertex $v$ by $\neighbours_{G}(v) $, that is the set of vertices $ w \in V $ such that $\{v,w\}\in E$.

Let $\chi$ be a locally finite subset of $\RR^d$ in generic position, i.e.\ such that each subset of size $n \leq d $ is affinely independent and no $d+2$ points lie on a sphere. For a $(d+1)$-tuple of points $ x_1 , \ldots , x_{d+1} \in \chi $, we denote by $ B ( x_{1:d+1} ) $ the open circumball associated with these points. We define a Delaunay edge between $x_i$ and $x_j$ for each $1\leq i, j\leq d+1$, $i\neq j$, when $\chi\cap B(x_{1:d+1})=\varnothing$, and denote by $\del(\chi)$ the set of these edges.

Let $ x_0 \in \chi $. With a slight abuse of notation, we denote  by $\neighbours_{\chi}(x_0) = \neighbours_{(\chi,\del(\chi))}(x_0) $ the set of neighbors of $x_0$ in the Delaunay graph associated with $\chi$. In particular, the degree of $x_0$ is $ d_{\chi}(x_0) = \# \neighbours_{\chi}(x_0) $. We also denote by $F_\chi ( x_0 ) $ the \textit{Voronoi flower at $ x_0 $}, defined as the union of all open balls which do not contain any point of $\chi$ and which are circumscribed to $ x_0 $ and $ d $ other points of $ \chi $, i.e.\ 
\[ F_\chi ( x_0 ) = \bigcup_{ \substack{ x_{1:d} \in \chi^d_{\neq} \\ B(x_{0:d}) \cap \chi =\emptyset } } B(x_{0:d}). \]
The Voronoi flower at $ x_0 $ only depends on its neighbors in the corresponding Delaunay graph.
Reciprocally the Voronoi flower at $x_0$ determines its set of neighbors.
We call \textit{$\Phi$-content} of $x_0$ the volume of its Voronoi flower and denote it by
\[ \Phi_\chi(x_0) = \vol{d}{F_\chi(x_0)} . \]
If $\chi$ is a finite subset $\{x_0, x_1, \ldots, x_k\}$ of $\RR^d$, with $k\geq d$, we use the shorter notation:
\[ F_{x_{1:k}}(x_0) = F_{\{ x_0, \ldots , x_k \}}(x_0) 
\quad \text{and}\quad  \Phi_{x_{1:k}}( x_0 ) = \vol{d}{F_{x_{1:k}}(x_0)} .\]
Finally, for each $B\in \mathcal{B}_{+}^d$ and $ k \in \NN $, we let 
\begin{equation}
    \label{eq:defnumberexceedances} 
    M^B_\chi = \max_{x\in \chi\cap B} d_{\chi}(x)  \quad \text{and} \quad 
    N_\chi^B[k]=\sum_{x\in \chi\cap B}\ind{d_\chi(x)\geq k}
\end{equation}
If $\chi\cap B=\varnothing$, we take $M^B_\chi=-\infty$.

\subsection{The typical degree} 
\label{sec:typicaldegree}
Recall that $\XX$ denotes a stationary Poisson point process of intensity 1 in $\RR^d$. To describe the mean behaviour of the Poisson-Delaunay graph, the notion of typical degree is introduced as follows. Let $B\in \mathcal{B}_{+}^d$ be fixed. The \textit{typical degree} is defined as the discrete random variable $\typ$ with distribution given by 
\begin{equation}\label{eq:deftypical}\PPP{\typ=k} = \frac{1}{\vol{d}{B}}\EEE{\sum_{x\in \XX\cap B}\ind{d_\XX(x)=k} },  \end{equation} for any integer $k$. It is clear that $\PPP{\typ=k} = 0$ for any $k\leq d$.  Due to the stationarity of $\XX$, it can be shown that the right-hand side does not depend on $B$. Thanks to the Mecke-Slivnyak theorem (e.g. Theorem 9.4 in \cite{LP}), it is well-known that  
\begin{equation}\label{eq:proptypical}\typ \overset{\rm d}{=} d_{\XX\cup\{0\}}(0), \end{equation}  where $\overset{\rm d}{=}$ denotes the equality in distribution. 

% \paragraph{An explicit integral representation of $\PPP{\typ=k}$}
\paragraph{Integral representation for the distribution of the typical degree} 
Let $x_{1:k} \in (\RR^d)^k_{\neq}$ be a $k$-tuple of distinct points, with $k\geq d$. We say that $x_{1:k}$  is in \textit{convex position} if $0$ is connected to all the $x_i$, $i=1,\ldots, k$, in the Delaunay graph associated with $\{0,x_1,\ldots,x_k\}$ and that $0$ is in the interior of the Voronoi flower $F_{\{0,x_1,\ldots,x_k\}}(0)$.
We denote by $C_k$ the set of all $k$-tuples of points in $\RR^d$ which are in convex position.
This set is stable under permutations, meaning that for any $x_{1:k}\in C_k$ and any permutation $\sigma$ of the set $\{1,\ldots,k\}$, we have  $x_{\sigma(1:k)} = (x_{\sigma(1)},\ldots,x_{\sigma(k)}) \in C_k$. We shall now derive an integral representation of the distribution of the typical degree $\typ$. 
\begin{Le}
\label{Le:explicitqk}
For each $k\geq 1$, we have
\[\PPP{\typ=k} = \int_{C_k}\ind{\Phi_{y_{1:k}}(0)\leq 1}\mathrm{d}y_{1:k}.\]
\end{Le}

\begin{proof}
Using the above notation and Equation \eqref{eq:proptypical}, we write
\[\PPP{\typ=k} = \frac{1}{k!}\EEE{\sum_{x_{1:k}\in\XX^k_{\neq}} \ind{x_{1:k}\in C_k} \ind{F_{x_{1:k}}(x_0)\cap\XX=\emptyset}}.\]
The multivariate Mecke equation (e.g.\ Theorem 4.4 in \cite{LP}) allows us to rewrite the expectation of a sum over $k$-tuples of points in a Poisson point process as an integral over $k$-tuples of points in $\RR^d$. Thanks to this formula, this gives  \begin{align*}\PPP{\typ=k} & = \frac{1}{k!}\int_{C_k}\PPP{F_{x_{1:k}}(x_0)\cap \left(\XX \cup\{x_1,\ldots, x_k\}\right)=\emptyset}\mathrm{d}x_{1:k}\\
 & = \frac{1}{k!}\int_{C_k}e^{-\Phi_{x_{1:k}}(0)}\mathrm{d}x_{1:k},
 \end{align*} 
 where the second line is also a consequence of the fact that $\eta$ is a Poisson point process. Using the fact that $ e^{ - t } = \int_{ t }^{ \infty } e^{ - s } \mathrm{d}s $, we get
    \begin{align*}
        \PPP{\typ=k}
        & = \frac{ 1 }{ k ! }  \int_{ C_k } \int_{ 0 }^{ \infty } \ind{ \Phi_{x_{1:k}} ( 0 ) \leq s } e^{ - s } \mathrm{d}s\mathrm{d}x_{1:k}. 
    \end{align*}
    Now since being in convex position is invariant under rescaling and since $ \Phi_{s^{1/d }y_{1:k}}(0) = s \Phi_{y_{1:k}} (0) $, the change of variables $ x_{1:k} = s^{ 1/d } y_{1:k} $ gives
    \begin{equation*}
        \PPP{\typ=k}
         = \frac{ 1 }{ k ! } \int_{ 0 }^{ \infty }  e^{ - s } \int_{ C_k}  \ind{ \Phi_{y_{1:k}} (0) \leq 1 } s^k  \dint y_{1:k}  \, \dint s = \int_{C_k}\ind{\Phi{y_{1:k}}(0)\leq 1}\mathrm{d}y_{1:k}.
    \end{equation*}
This concludes the proof of Lemma \ref{Le:explicitqk}. 
    \end{proof}
    
 \paragraph{Estimates for the distribution of the typical degree}
 The following result provides bounds for the distribution of the typical degree in $\RR^d$, $d\geq 2$. 
 \begin{Prop}
    \label{Prop:Gilles}
  There exist two constants $c_1,c_2>0$ depending on $d$ such that, for $k$ large enough, we have 
\begin{enumerate}[(i)]
\item $\PPP{\typ=k} \leq c_2\,  k^{-\frac{2}{d-1}}\PPP{\typ=k-1}$,\label{th:upperbound}
\item $\PPP{\typ=k} \geq c_1^k \, k^{\frac{-2}{d-1}k}$.\label{th:lowerbound}
\end{enumerate}    
    In particular, for some constant $c_3$, we have 
    \[ c_1^k \, k^{\frac{-2}{d-1}k} \leq \PPP{\typ=k}    \leq c_3^k\, k^{-\frac{2}{d-1} k}\]
and 
\begin{equation}
    \label{eq:boundsGtypical}
\PPP{\typ\geq k}\eq[k]{\infty}\PPP{\typ =k}.
  \end{equation} 	    
    
  \end{Prop} 
Proposition \ref{Prop:Gilles} is very similar to two results in \cite{BCR} (Theorem 1.1 and Theorem 1.2) in which  estimates for the distribution of the typical number of facets in a Poisson hyperplane tessellation are given. We do not give its proof because it relies on a simple adaptation of several arguments included in \cite{BCR} to our setting. However, for a complete proof of Proposition \ref{Prop:Gilles}, we refer the reader to Chapter 5 in \cite{BonnetPhD} (Theorems 5.1.7 and 5.5.1). Besides, according to Proposition \ref{Prop:Gilles},  the distribution of the typical degree belongs to the class of discrete distributions considered by Anderson \cite{And}. Roughly speaking, this explains why the maximal degree belongs to two consecutive integers when the size of the window goes to infinity.

In the particular setting $d=2$,  a more precise estimate of the distribution of the typical degree is established by Hilhorst (see Equation (1.2) in \cite{Hil}):
\begin{equation}\label{eq:hilhorst} \PPP{\typ=k} = \frac{C}{4\pi^2}\cdot \frac{(8\pi^2)^k}{(2k)!}\left( 1 + O\left(  k^{-\frac{1}{2}} \right)\right),
\end{equation} where $C\simeq 0.34$. The above result is extended by Calka and Hilhorst for a larger class of random polygons in $\RR^2$ (see Equation (1.5) in \cite{HC}). However, as opposed to Proposition \ref{Prop:Gilles}, their result is not enough to derive Theorem \ref{Th:maxdegree} because it does not provide a recurrence relation between $\PPP{\typ=k}$ and $\PPP{\typ=k-1}$. 
The following remark presents a heuristic argument suggesting that, in the case $d=2$, a careful study based on \eqref{eq:hilhorst} should refine the estimate $\EEE{\Delta_\rho}$.

\begin{Rk}
The estimate $\EEE{\Delta_\rho} \simeq \frac{d-1}{2} \cdot \frac{\log \rho}{\log \log \rho}$ seems to be accurate only for extremely high values of $\rho$.
Indeed, for $d=2$, Figure \ref{fig:Maxdegree} illustrates that the empirical distribution of $\Delta_{10^6}$ concentrates on $15$ and $16$ rather than around $\frac{1}{2} \cdot \frac{\log 10^6}{\log \log 10^6} \simeq 2.6$.
This is not surprising because of the extremely slow growth of the logarithm.

Nevertheless, the following heuristic argument provides a much closer estimate of $\Delta_{10^6}$.
Thanks to \eqref{eq:hilhorst} and because of the extremely fast decay of this expression as $k$ grows, we have
\[\PPP{\typ\leq 13}\simeq 1-\PPP{\typ = 14} \simeq 1- 10^{-5} \quad  \text{and} \quad \PPP{\typ \leq 15} \simeq 1-\PPP{\typ = 16} \simeq 1-7.6\cdot 10^{-8}.\]
Assuming that the maximal degree has the same behaviour as the maximum of $10^6$ independent random variables with the same distribution as the typical degree,
\[ \PPP{\Delta_{10^6} \leq 13} \simeq 4\cdot 10^{-5} \quad  \text{and} \quad \PPP{\Delta_{10^6}\leq 15}\simeq 0.93 .\]
This suggests that $\Delta_{10^6} \in \{ 14 , 15 \}$ with high probability, which is almost what we observe in Figure \ref{fig:Maxdegree}. In the setting $d=2$, a careful study based on \eqref{eq:hilhorst} should provide an estimate of $\EEE{\Delta_\rho}$ which fits the correct value faster than ours.
\end{Rk}

  \subsection{The function $\rho \mapsto I_\rho$}
In this section, we define a function $\rho\mapsto I_\rho$, $\rho>0$, with values in $\NN$, and which depends on the dimension $d$. When $d=2$, this is the function appearing in Theorem \ref{Th:maxdegree}. To define $I_\rho$ for any $d\geq 2$, our approach is mainly inspired from \cite{And}. For any $k\geq d+1$, let $G(k)=\PPP{\typ>k}$, where $\typ$ is the typical degree in $\RR^d$. We extend $G$ as a continuous function as follows. For any $k\geq d+1$, we let $h(k)=-\log G(k)$. We consider an auxiliary function $h_c$ defined as the extension of $h$ obtained by linear interpolation, i.e.\ for any $x\geq d+1$,
\[h_c(x) = h(\lfloor x\rfloor ) + (x-\lfloor x\rfloor)(h(\lfloor x+1\rfloor) - h(\lfloor x\rfloor)).\] The function $h_c$ is continuous, strictly increasing  and $\lim_{x\rightarrow\infty}h_c(x) = \infty$.  Then we extend $G$ as the continuous function $G_c(x)=e^{-h_c(x)}$ for each $x\geq d+1$. In particular, $G_c$ is a continuous strictly decreasing function.
%This allows us to define the following quantity 
  %\begin{equation}
  %\label{def:An} A_\rho:=G_c^{-1}\left(\frac{1}{\rho} \right).
  %\end{equation}
  Now, we define the function $\rho\mapsto I_\rho$, $\rho>0$, by 
  \begin{equation}
  \label{def:In}
  I_\rho=\left\lfloor G_c^{-1}\left(\frac{1}{\rho} \right)+\frac{1}{2}\right\rfloor.
  \end{equation}

   \section{Intermediate results}   
\label{sec:intermediateresults}
 In this section, we establish intermediate results which will be used in the proofs of Theorems \ref{Th:maxdegree} and \ref{Th:maxdegreed}. 

\subsection{Technical results}  
The following lemma provides the exact order of $I_\rho$.

  \begin{Le}
    \label{Le:eqAn}
   Let $d\geq 2$ and let $I_\rho$ be as in \eqref{def:In}. Then 
    \[I_\rho\eq[\rho]{\infty} \frac{d-1}{2} \cdot \frac{\log \rho}{\log \log \rho}.\]
  \end{Le}
  
  The following lemma deals with the probability that the typical degree is larger than $I_\rho$ up to an additive constant.
  
  \begin{Le}
    \label{Le:tail}
    Let $d\geq 2$ and let $I_\rho$ be as in \eqref{def:In}. Then
    \begin{enumerate}[(i)]
      \item $\rho\PPP{\typ\geq I_\rho+2}\conv[\rho]{\infty}0$;\label{Le:tail1}
      \item $ \left(\frac{\log\log \rho}{\log \rho}\right)^{\frac{2l}{d-1}} \rho\PPP{\typ\geq I_\rho-l} \conv[\rho]{\infty}\infty$ for each $l\geq 0$.\label{Le:tail3}
    \end{enumerate}
  \end{Le}
  In particular, when $l=0$, Lemma \ref{Le:tail} \eqref{Le:tail3} means that  $\rho\PPP{\typ\geq I_\rho}$ converges to infinity. By adapting the proof of Lemma \ref{Le:tail}, it  can also be shown that $\rho\PPP{\typ\geq I_\rho+1}$ does not converge as $\rho$ goes to infinity because its infimum and supremum limits equal 0 and $\infty$ respectively. 
  
  As a consequence of Lemma \ref{Le:tail}, we could show that if $X_1,\ldots, X_n$ is a sequence of $n$ independent and identically distributed random variables, with the same distribution as the typical degree, then the maximum of $X_1,\ldots, X_n$ belongs to $\{I_n, I_n+1\}$ with probability tending to 1 as $n$ goes to infinity. 
  Even if the independency is lost, as it is the case with the vertices' degrees, the maximum remains upper bounded with high probability by $I_n+1$.
  On the other hand, if the dependency is too strong it is impossible to give a non-trivial lower bound.
  Therefore the proofs of Theorems \ref{Th:maxdegree} and \ref{Th:maxdegreed} rely on a quantification of the dependencies between the vertices' degrees.
  In Section \ref{sec:discretization} we show that vertices which are far enough have almost independent degrees.
  This is sufficient to derive Theorem \ref{Th:maxdegreed}.
  In Section \ref{sec:cluster}, at the cost of reducing the setting to $d=2$, we deal with a more local scale by showing that {there is no 5-tuple of nodes which are close to each others and such that their degrees are simultaneously larger than $I_\rho$.}  This is one of the greatest difficulties treated in this paper and one of the key arguments to prove Theorem \ref{Th:maxdegree}.
  
%   Such a remark will be used to derive Theorem \ref{Th:maxdegree}. Besides, Lemma \ref{Le:tail} is sufficient to derive Theorem  \ref{Th:maxdegreed}. On the opposite, it is not enough for the proof of Theorem \ref{Th:maxdegree} because our main theorem requires to deal with the dependence between the cells. The key ingredient to deal with this dependence is discussed below.

\subsection{A subdivision of the window $\window$} 
\label{sec:discretization}
It is well-known that a Poisson-Delaunay graph in $\RR^d$ has good mixing properties. To capture this property, we proceed as follows. We partition $\window = \rho^{1/d}[0,1]^d$  into a set $\mathcal{V}_\rho$ of $\mathrm{N}_\rho^{d}$ closed sub-cubes of equal size, where  \begin{equation}\label{def:Nn}N_\rho:=\left\lfloor  \left(\frac{\rho}{\alpha\log \rho}\right)^{1/d}  \right\rfloor,
\end{equation} for some $\alpha >2$. The volume of each sub-cube is approximately $\alpha\log \rho$ as $\rho$ goes to infinity. The sub-cubes are indexed
by the set of $\mathbf{i}:=(i_{1},\ldots ,i_{d})\in \left\{ 1,\ldots, N_\rho\right\} ^{d}$. With a slight abuse of notation, we identify a cube with its index. We denote by $\mathbf{i}_0$ the unique sub-cube in $\mathcal{V}_\rho$ which contains the origin. We now introduce a distance between sub-cubes $\mathbf{i}$ and $\mathbf{j}$ as $d(\mathbf{i},\mathbf{j}):=\max_{1\leq s\leq d}|i_{s}-j_{s}|$. If $\mathcal{I}$ and $\mathcal{J}$ are two sets of sub-cubes, we let \[d(\mathcal{I},\mathcal{J}):=\min_{\mathbf{i}\in \mathcal{I},\mathbf{j}\in \mathcal{J}}d(\mathbf{i},\mathbf{j}).\] For any $A\subset \RR^d$, we define 
\[\mathcal{I}(A):=\{\mathbf{i}\in \mathcal{V}_\rho: \mathrm{int}(\mathbf{i}\cap A)\neq \varnothing\}.\] 
Finally, to ensure several independence properties, we introduce the event 
\begin{equation*}
\mathscr{E}_\rho:=\cap _{\mathbf{i}\in \mathcal{V}_\rho}\{\XX \cap \mathbf{i}\neq
\varnothing \}.
\end{equation*}
The event $\mathscr{E}_\rho$ is extensively used in stochastic geometry to derive central
limit theorems or limit theorems in Extreme Value Theory (see e.g.\ \cite{AB,Chen}). It will
 play a crucial role in the rest of the paper. The following lemma captures the idea of ``local dependence''.  

\begin{Le}
\label{Le:Arho} Let $A,B\subset \window$ and let $N_\rho$ as in \eqref{def:Nn}, with $\alpha>2$. Then
\begin{enumerate}[(i)]
\item \label{Le:Arho1} conditional on the event $\mathscr{E}_\rho$, the random variables $M^A_\XX$ and $M^B_\XX$ are independent when $d(\mathcal{I}(A),\mathcal{I}(B))>D$, where $D:=4(\lfloor \sqrt{d}\rfloor +1)$;
\item \label{Le:Arho2}  $\PPP{\mathscr{E}_\rho^{c}} = O\left(\rho^{-(\alpha-1)} \right)$.
\end{enumerate}
\end{Le}
The above lemma has been used in various papers (e.g. Lemma 5 in \cite{Chen} and Lemma 1 in \cite{ChenRob}). We refer the reader to these papers for a proof.

 \subsection{Family of five nodes with large degrees when $d=2$}\label{sec:cluster}
 In this section, we only deal with the case $d=2$. Recall that $\mathbf{i}_0$ is defined as the unique square in $\mathcal{V}_\rho$ which contains the origin (see Section \ref{sec:discretization}). When $\rho$ goes to infinity, the order of the area of such a square is $\alpha \log \rho$, with $\alpha>2$. 
 
 The following result shows that, with high probability, there is no 5-tuple of nodes which are close to each others and such that their degrees are simultaneously larger than $I_\rho$.
 Recall that the random variable $ N_\XX^B[k]=\sum_{x\in \XX\cap B}\ind{d_\XX(x)\geq k} $, as introduced in \eqref{eq:defnumberexceedances}, denotes the number of exceedances in $B$. 
 \begin{Prop}
 \label{Prop:cluster5}
 Let $B\in \mathcal{B}_{+}^2$ be  a convex and symmetric Borel subset in $\RR^2$. Then there exist two positive constants $c_4,c_5$ (independent of $B$) such, that for any $k\geq 1$, we have
%\begin{equation} \label{def:Cn}C_n= \left\{\sum_{x\in \XX\cap i_0}\ind{d_\XX(x)> I_\rho-1}\geq 5\right\}.
%\end{equation}
 \[\PPP{N_\XX^B[k]\geq 5}  \leq \left(c_4\, \vol{2}{B} \, k^{-2}+c_5^k\, \vol{2}{B}^{2}\, k^{-2k/23}\right)\PPP{\typ=k}.\]
 \end{Prop}
The above result is the key ingredient to derive Theorem \ref{Th:maxdegree} and contains the main difficulty of our problem. It extensively uses the fact that the Delaunay graph is planar.

\subsection{A lower bound for the distribution's tail of the maximal degree in a block}
According to \eqref{eq:deftypical}, it is clear that $\PPP{M^{B}_\XX \geq k}\leq \vol{d}{B}\PPP{\typ \geq k}$ for each $k\in\NN$ and any Borel set $B\subset\RR^d$. The following results deal with the reciprocal of this inequality. The first one only concerns the case $d=2$ and will be used to prove Theorem \ref{Th:maxdegree}. 

\begin{Prop}
\label{Prop:maxblock}
Let $B\in \mathcal{B}_{+}^2$ be a convex and symmetric Borel subset in $\RR^2$. Then there exists an integer $k_0$ depending on $B$ such that
 \[\PPP{M^{B}_\XX \geq k} \geq \frac{\vol{2}{B}}{5} \PPP{\typ \geq k},\]
 for any $k\geq k_0$.
\end{Prop}

The following result provides a lower bound which is less accurate than the one of Proposition \ref{Prop:maxblock}, but deals with the general case $d\geq 2$. It will be used to prove Theorem  \ref{Th:maxdegreed}.

\begin{Prop}
\label{Prop:maxblockd}
Let $B\in \mathcal{B}_{+}^d$, $d\geq 2$. Then, for any $k\in\NN$ and $h\geq 1$, we have
 \[\PPP{M^{B}_\XX \geq k} \geq \frac{\vol{d}{B}}{h+1}\left( \PPP{\typ \geq k} - \exp\left(-\vol{d}{B}\left(1-e+\frac{h}{\vol{d}{B}}\right)\right)\right).\]
\end{Prop}

\subsection{A bound for the probability of a finite union of events}
\begin{Le}
\label{Le:union}
   Let $(\Omega, \mathcal{F}, \PP)$ be a probability space and let $ B^{(1)}, \ldots, B^{(K)}$, $K\geq 1$, be a collection of events such that $ \PPP{ \bigcap_{ j \leq k + 1 } B^{( i_j) } } = 0 $, for any $ 1 \leq  i_1 < \ldots < i_{ k + 1 }\leq K $.    Then 
    \begin{align*}
 \PPP{ \bigcup_{ i= 1}^K B^{(i)} }\geq \frac{1}{k}        \sum_{ i = 1}^K \PPP{ B^{(i)} }. 
    \end{align*}
\end{Le}
Notice that when $k=1$, the inequality  is actually an equality.

\section{Proofs of Theorems and Corollaries}
\label{sec:maintheorem}
\subsection{Proof of Theorem \ref{Th:maxdegree}}
Let $I_\rho$ be as in \eqref{def:In}. According to Lemma \ref{Le:eqAn}, we have  $I_\rho\eq[\rho]{\infty}\frac{1}{2}\cdot \frac{\log \rho}{\log\log \rho}$. Now, we have to show that $\PPP{\Delta_\rho \in \{I_\rho, I_\rho+1\}}\conv[\rho]{\infty}1$.

 To do it, we first notice that
\begin{align*}
\PPP{\Delta_\rho\geq I_\rho+2} & = \PPP{\bigcup_{x\in \XX\cap \window}\left\{ d_\XX(x)\geq I_\rho+2  \right\}}\\
& \leq \EEE{\sum_{x\in \XX\cap \window}\ind{ d_\XX(x)\geq I_\rho+2  }}\\
& = \rho\PPP{\typ \geq I_\rho+2},
\end{align*}
where the last equality comes from \eqref{eq:deftypical}. According to Lemma \ref{Le:tail}, the probability $\PPP{\Delta_\rho\geq I_\rho+2}$ converges to 0.

Secondly, we show that $\PPP{\Delta_\rho\leq I_\rho-1}$ converges to 0, which will prove Theorem \ref{Th:maxdegree}. This second step is much more delicate than the first one and uses all the intermediate results presented in the previous section. We subdivide the window $\window$ into $N_\rho^2$ sub-squares of equal size as described in Section \ref{sec:discretization}, where $N_\rho$ is defined in \eqref{def:Nn} with $d=2$, $\alpha>2$. This gives
\begin{equation}
\label{eq:majDelta0}
\PPP{\Delta_\rho \leq I_\rho-1}  \leq \PPP{\Delta_\rho\leq I_\rho-1|\mathscr{E}_\rho} + \PPP{\mathscr{E}_\rho^c} = \PPP{\left.\bigcap_{\mathbf{i}\in \mathcal{V}_\rho}  \left\{M_\XX^{\mathbf{i}}\leq I_\rho-1\right\}  \right|\mathscr{E}_\rho} + \PPP{\mathscr{E}_\rho^c},
\end{equation}
where $M_\XX^{\mathbf{i}}$ is defined in Section \ref{sec:notation} and where $\mathscr{E}_\rho$ and $\mathcal{V}_\rho$ are defined in Section \ref{sec:discretization} respectively. Let $\mathcal{V}'_\rho$ be the family of sub-cubes $\mathbf{i}$ for which each coordinate of its index $(i_1,i_2)$ is a multiple of $9$. The cardinality of $\mathcal{V}'_\rho$ is larger than $\lfloor N_\rho / 9 \rfloor^2$. Note that for any pair of distinct sub-cubes $\mathbf{i}, \mathbf{j}\in \mathcal{V}'_\rho$, we have  $d(\mathbf{i}, \mathbf{j})>D$, with $D=4(\lfloor\sqrt{2}\rfloor +1)=8$. According to Lemma \ref{Le:Arho} \eqref{Le:Arho1}, conditional on $\mathscr{E}_\rho$, we know that the events $\{M_\XX^{\mathbf{i}}\leq I_\rho-1\}$ and $\{M_\XX^{\mathbf{j}}\leq I_\rho-1\}$ are independent for each $\mathbf{i}\neq  \mathbf{j}\in \mathcal{V}'_\rho$. Using the fact that $\PPP{\left.M_\XX^{\mathbf{i}}\leq I_\rho-1\right|\mathscr{E}_\rho}\leq \PPP{M_\eta^{\mathbf{i}_0}\leq I_\rho-1 }/\PPP{\mathscr{E}_\rho}$ for each $\mathbf{i}$, we get
\begin{equation*} \PPP{\left.\bigcap_{\mathbf{i}\in \mathcal{V}_\rho}  \left\{M_\XX^{\mathbf{i}}\leq I_\rho-1\right\}  \right|\mathscr{E}_\rho} \leq  \prod_{\mathbf{i}\in \mathcal{V}'_\rho}\PPP{\left.M_\XX^\mathbf{i}\leq I_\rho-1\right|\mathscr{E}_\rho}   \leq   \left( \frac{ \PPP{M^{\mathbf{i}_0}_\XX  \leq I_\rho-1}}{\PPP{\mathscr{E}_\rho}}  \right)^{\lfloor N_\rho / 9 \rfloor^2}.
\end{equation*}
Besides, according to Lemma \ref{Le:Arho} \eqref{Le:Arho2}, we know that $\PPP{\mathscr{E}_\rho} = 1-O(\rho^{-(\alpha-1)})$ with $\alpha>2$. This together with the above equation and Equation \eqref{eq:majDelta0} gives, for $\rho$ large enough,
\begin{equation*}
\PPP{\Delta_\rho \leq I_\rho-1}  \leq \exp\left(-\frac{N_\rho^2}{82} \left(\PPP{M^{\mathbf{i}_0}_\XX\geq I_\rho}-\log(\PPP{\mathscr{E}_\rho})\right)\right) + O\left(\rho^{-(\alpha-1)}\right).
\end{equation*}
Since $\alpha >2$, we have $N_\rho^2 \log(\PPP{\mathscr{E}_\rho})\conv[\rho]{\infty}0$. Thus 

\begin{equation}
\label{eq:majDelta}
\PPP{\Delta_\rho \leq I_\rho-1}  \leq \exp\left(-\frac{N_\rho^2}{82} \PPP{M^{\mathbf{i}_0}_\XX\geq I_\rho}+o(1)\right) + O\left(\rho^{-(\alpha-1)}\right).
\end{equation}

To prove that the right-hand side of the above equation converges to 0, we have to show that $N_\rho^2 \PPP{M^{\mathbf{i}_0}_\XX\geq I_\rho}$ converges to infinity. Notice that we cannot directly apply Proposition \ref{Prop:maxblock} to the block $B=\mathbf{i}_0$ and to the integer $k=I_\rho$ because both quantities depend on $\rho$. To deal with $\PPP{M^{\mathbf{i}_0}_\XX\geq I_\rho}$,  we sub-divide the square $\mathbf{i}_0$ into $K_\rho^2$  sub-squares of equal size, say $S_1, \ldots, S_{K_\rho^2}$, with $K_\rho=\lfloor (\log \rho)^{1/2} \rfloor$. The area of each sub-square $S_i$ is larger than $\alpha$, and converges to $\alpha$ as $\rho$ goes to infinity.

Now, we define a finite collection of events $B_\rho^{(1)}, \ldots, B_\rho^{(K_\rho^2)}$ as follows. For each $1\leq i\leq K_\rho^2$, we let
\[B_\rho^{(i)}=\{M^{S_i}_\XX \geq I_\rho\} \setminus \{N_\XX^{\mathbf{i}_0}[I_\rho]\geq 5\},\] where we recall that 
$N_\XX^{\mathbf{i}_0}[I_\rho]$, as defined in \eqref{eq:defnumberexceedances}, denotes the number of nodes with degree larger than $I_\rho$. In particular, we have
\[\PPP{M^{\mathbf{i}_0}_\XX\geq I_\rho} \geq \PPP{\bigcup_{i=1}^{K_\rho^2}B_\rho^{(i)}}.\]
Moreover, we know that    $ \PPP{ \bigcap_{ j \leq 5 } B^{( i_j) }_\rho } = 0 $, for any $ 1 \leq  i_1 < \ldots < i_{ 5 }\leq K_\rho^2 $. It follows from  Lemma \ref{Le:union} that 
\[\PPP{M^{\mathbf{i}_0}_\XX\geq I_\rho}\geq  \frac{1}{4}\sum_{i=1}^{K_\rho^2}\PPP{B_\rho^{(i)}} \geq \frac{1}{4}\sum_{i=1}^{K_\rho^2}\PPP{M_\XX^{S_i}\geq I_\rho} - \frac{K_\rho^2}{4}\PPP{N_\XX^{\mathbf{i}_0}[I_\rho]\geq 5}.\]
According to Proposition \ref{Prop:cluster5} and the facts that $\vol{2}{\mathbf{i}_0}=O\left(\log \rho\right)$ and $I_\rho\eq[\rho]{\infty}\frac{1}{2}\cdot \frac{\log \rho}{\log\log \rho}$, we have \[\PPP{N_\XX^{\mathbf{i}_0}[I_\rho]\geq 5} = o\left(\PPP{\typ=I_\rho}\right).\] 
To deal with $\PPP{M_\XX^{S_i}\geq I_\rho}$, recall that the area of  $S_i$ is larger than $\alpha$ so that, up to a translation, the square $S_i$ contains the square $S'=\left[-\sqrt{\frac{\alpha}{2}}, \sqrt{\frac{\alpha}{2}}\right]^2$. Due to the stationarity of $\XX$, we have $\PPP{M_\XX^{S_i}\geq I_\rho}\geq \PPP{M_\XX^{S'}\geq I_\rho}$. According to Proposition \ref{Prop:maxblock} applied to $B=S'$, with $\vol{2}{S'} = \alpha$,  we obtain for $\rho$ large enough and for each $1\leq i\leq K_\rho^2$, 
\[\PPP{M_\XX^{S_i}\geq I_\rho} \geq \frac{\alpha}{5}\PPP{\typ\geq I_\rho}.\]
Summing over $i=1,\ldots, K_\rho^2$, we deduce for $\rho$ large enough that 
\[\PPP{M^{\mathbf{i}_0}_\XX\geq I_\rho} \geq \frac{K_\rho^2\alpha}{21}\PPP{\typ\geq I_\rho}.
\]

Now, we can conclude the proof of Theorem \ref{Th:maxdegree}. Indeed, it follows from the above inequality and Equation \eqref{eq:majDelta} that
\[\PPP{\Delta_\rho\leq I_\rho-1} \leq \exp\left(-c\, N_\rho^2 \, K_\rho^2\PPP{\typ\geq I_\rho} + o(1) \right) + O\left( \rho^{-(\alpha-1)}  \right),\]
 for some positive constant $c$. This proves Theorem \ref{Th:maxdegree} thanks to Lemma \ref{Le:tail}  \eqref{Le:tail3} and the fact that $N_\rho^2\eq[\rho]{\infty}\frac{\rho}{\alpha\log \rho} $ and $K_\rho^2\eq[\rho]{\infty}\log \rho$.

\subsection{Proof of Corollary \ref{cor:onevalue}}
To define the sequence $(\rho_i)$, we first introduce for each $i\geq 1$ the set $D_i=\{\rho\in \RR_+: I_\rho=i\}$, where $I_\rho$ is as in \eqref{def:In}. Let $i\geq 1$ be fixed. The set $D_i$ is non-empty since it contains the number $m_i=\left(G_c\left(i-\frac{1}{2}\right)\right)^{-1}$. Because $\rho\mapsto I_\rho$ is increasing,  $D_i$ is an interval. Moreover, this interval is  bounded since $D_j$ is non-empty for each $j\geq 1$. Thus the family $(D_i)$ is a partition of $\RR_+$ into bounded intervals. We can easily show that these intervals are left-closed and right-open respectively. 

Now, we define the sequence $(\rho_i)$ as follows. For each $i\geq 2$, we let $\rho_i=\sup D_{i-1}=\min D_{i}$.  In particular, we have $I_{\rho_i-1} \leq I_{\rho_i}-1=i-1$. The sequence $(\rho_i)$ is increasing and converges to infinity. According to Theorem \ref{Th:maxdegree}, we have
\begin{equation}
\label{eq:coronevalue}
\PPP{\Delta_{\rho_i} \in \{I_{\rho_i}, I_{\rho_i}+1\}}\conv[i]{\infty}1.
\end{equation}
Moreover, according to Lemma \ref{Le:tail} \eqref{Le:tail1}, we know that $(\rho_i-1)\PPP{\typ\geq I_{\rho_i-1}+2}$ converges to 0 as $i$ goes to infinity. Since $I_{\rho_i-1} \leq I_{\rho_i}-1$, this implies that 
\[\rho_i\PPP{\typ\geq I_{\rho_i}+1}\conv[i]{\infty}0.\] Bounding $\PPP{\Delta_{\rho_i}\geq I_{\rho_i}+1}$ by $\rho_i\PPP{\typ\geq I_{\rho_i}+1}$, we deduce that $\PPP{\Delta_{\rho_i}\geq I_{\rho_i}+1}$ converges to 0. This together with  \eqref{eq:coronevalue} and the fact that $I_{\rho_i}=i$ concludes the proof of Corollary \ref{cor:onevalue}.

  \subsection{Proof of Theorem \ref{Th:maxdegreed}}
  Let $J_\rho=I_\rho+1-l_d$, where $I_\rho$ is defined in \eqref{def:In}. According to Lemma \ref{Le:eqAn}, we know that  $J_\rho\eq[\rho]{\infty}\frac{d-1}{2}\cdot \frac{\log \rho}{\log\log \rho}$. Now, we have to show that $\PPP{\Delta_\rho \in \{J_\rho, J_\rho+1, \ldots, J_\rho+l_d\}}\conv[\rho]{\infty}1$.
  
  As in the proof of Theorem \ref{Th:maxdegree} we easily show that $\PPP{\Delta_\rho\geq I_\rho+2}$ converges to 0 as $\rho$ goes to infinity. It remains to prove that  $\PPP{\Delta_\rho\leq I_\rho-l_d}$ also converges to 0. To do it, we proceed at this step in the same spirit as in the case $d=2$. We divide $\window$ into $N_\rho^d$ sub-cubes of equal size, where $N_\rho$ is given in \eqref{def:Nn}, for some $\alpha>2$. For some positive constant $c$ this gives (see Equation \eqref{eq:majDelta})
  \begin{equation}
  \label{eq:ineqmaxdegreed1} 
  \PPP{\Delta_\rho\leq I_\rho-l_d}  \leq \exp\left(-c\, N_\rho^d \PPP{M^{\mathbf{i}_0}_\XX\geq I_\rho-l_d+1}+o(1)\right) + O\left( \rho^{-(1-\alpha)}\right).
  \end{equation}
 Now we have to show that  $N_\rho^d \PPP{M^{\mathbf{i}_0}_\XX\geq I_\rho-l_d+1}$ converges to infinity. This time we apply Proposition \ref{Prop:maxblockd} by taking $B=\mathbf{i}_0$, $k=I_\rho-l_d+1$,  and $h=\beta\, \vol{d}{\mathbf{i}_0}$ for some $\beta>0$. This gives
 \begin{equation*}
 \PPP{M^{\mathbf{i}_0}_\XX\geq I_\rho-l_d+1} \geq \frac{\vol{d}{\mathbf{i}_0}}{\beta\,\vol{d}{\mathbf{i}_0}+1} \left(\PPP{\typ \geq I_\rho-l_d+1} -  \exp\left(-(1- e+\beta)\vol{d}{\mathbf{i}_0} \right)\right).
 \end{equation*}
To deal with  the right-hand side, we recall that $N_\rho^d\eq[\rho]{\infty}\frac{\rho}{\alpha\log \rho}$ and that $\vol{d}{\mathbf{i}_0}\eq[\rho]{\infty}\alpha\log \rho$, with $\alpha>2$. This gives 
\[N_\rho^d\, \exp\left(-(1- e+\beta)\vol{d}{\mathbf{i}_0}   \right) = o\left( \rho^{-(\alpha(1-e+\beta)-1)}  \right).\]
Taking $\beta$ in such a way that $\alpha(1-e+\beta)>1$, we obtain that $N_\rho^d\, \exp\left(-(1- e+\beta)\vol{d}{\mathbf{i}_0}   \right) $ converges to 0. Moreover, it follows from  Lemma  \ref{Le:tail}  \eqref{Le:tail3} that 
 \[    \left( \frac{\log\log \rho}{\log \rho}  \right)^{\frac{2(l_d-1)}{d-1}}\, \log (\rho)\, N_\rho^d \PPP{\typ \geq I_\rho-l_d+1} \conv[\rho]{\infty}\infty. \]
Thus  
  \begin{equation}
  \label{eq:ineqmaxdegreed2} 
 \left( \frac{\log\log \rho}{\log \rho}  \right)^{\frac{2(l_d-1)}{d-1}}\, \log (\rho)\, N_\rho^d \PPP{M^{\mathbf{i}_0}_\XX\geq I_\rho-l_d+1}\conv[\rho]{\infty}\infty.
   \end{equation}
  Since $\frac{2(l_d-1)}{d-1}>1$, we deduce that $N_\rho^d \PPP{M^{\mathbf{i}_0}_\XX\geq I_\rho-l_d+1}$ converges to infinity. This  concludes the proof of Theorem \ref{Th:maxdegreed}. 
 
\subsection{Proof of Corollary \ref{cor:eq}} 
 First, we write the expectation of the maximal degree as follows:
 \[\EEE{\Delta_\rho} = \sum_{k=1}^{I_\rho-l_d}k\PPP{\Delta_\rho=k} +  \sum_{k=I_\rho+1-l_d}^{I_\rho+d}k\PPP{\Delta_\rho=k} +  \sum_{k=I_\rho+d+1}^{\infty}k\PPP{\Delta_\rho=k}.\]
 
 For the first term, we notice that
 \[\sum_{k=1}^{I_\rho-l_d}k\PPP{\Delta_\rho=k} \leq \sum_{k=1}^{I_\rho-l_d}k\PPP{\Delta_\rho\leq I_\rho-l_d}\eq[\rho]{\infty}\frac{I_\rho^2}{2}\PPP{\Delta_\rho\leq I_\rho-l_d}.\]
 According to \eqref{eq:ineqmaxdegreed1} and \eqref{eq:ineqmaxdegreed2}   and the fact that $I_\rho\eq[\rho]{\infty}\frac{d-1}{2}\cdot\frac{\log \rho}{\log\log \rho}$, the term $\sum_{k=1}^{I_\rho-l_d}k\PPP{\Delta_\rho=k} $ converges to 0. Moreover, as a consequence of Theorem \ref{Th:maxdegreed}, we know that $\sum_{k=I_\rho+1-l_d}^{I_\rho+d}k\PPP{\Delta_\rho=k}$ is asymptotically equivalent to $\frac{d-1}{2}\cdot \frac{\log \rho}{\log\log \rho}$. For the third term, we have
 \[\sum_{k=I_\rho+d+1}^{\infty}k\PPP{\Delta_\rho=k} = (I_\rho+d)\PPP{\Delta_\rho\geq I_\rho+d+1} +  \sum_{k=I_\rho+d+1}^\infty\PPP{\Delta_\rho\geq k}.\]
The first term of the right-hand side can be bounded as follows: \[(I_\rho+d)\PPP{\Delta_\rho\geq I_\rho+d+1}\leq (I_\rho+d)\rho\PPP{\typ\geq I_\rho+d+1}.\]
According to Proposition \ref{Prop:Gilles},there exists a positive constant $c$ such that  
 \[(I_\rho+d)\PPP{\Delta_\rho\geq I_\rho+d+1}\leq c\, I_\rho^{-1}\, \rho\PPP{\typ=I_\rho+2}.\]
 The last term converges to 0 according to Lemma \ref{Le:tail} \eqref{Le:tail1}. Moreover, thanks again to Proposition \ref{Prop:Gilles}, we can also show that  the series $\sum_{k=I_\rho+d+1}^\infty\PPP{\Delta_\rho\geq k}$ is asymptotically equivalent to $\rho\PPP{\typ\geq I_\rho+d+1}$. Since this quantity converges to 0, this shows that  $\sum_{k=I_\rho+d+1}^{\infty}k\PPP{\Delta_\rho=k}\conv[\rho]{\infty}0$. Consequently, we have $\EEE{\Delta_\rho}\eq[\rho]{\infty}\frac{d-1}{2}\cdot \frac{\log \rho}{\log\log \rho}$. 
 
  \section{Proofs of technical results}
  \label{sec:lemmas}
  
  %Since $\PPP{\typ \geq k} = \PPP{\typ=k}+\sum_{l=1}^\infty \PPP{\typ=k+l}$, it is enough to show that $\sum_{l=1}^\infty \PPP{\typ=k+l} = o(\PPP{\typ=k})$.  To do it, we apply Proposition \ref{Prop:Gilles}. By induction, this gives for each $k,l\geq 1$:
%\[\PPP{\typ=k+l}\leq \overset{\sim}{c_2}^l \cdot \left(\frac{(k+l)!}{(k+1)!}\right)^{-\frac{2}{d-1}}\cdot (k+1)^{-\frac{2}{d-1}}\PPP{\typ=k}\leq \overset{\sim}{c_2}^l\cdot (l!)^{-\frac{2}{d-1}}\cdot (k+1)^{-\frac{2}{d-1}}\PPP{\typ=k}.\]
%Summing over $l\geq 1$, this shows that \[\sum_{l=1}^\infty \PPP{\typ=k+l}  = O\left(  (k+1)^{-\frac{2}{d-1}}\cdot \PPP{\typ=k} \right). \]

  \subsection{Proof of Lemma \ref{Le:eqAn}}
Let \begin{equation}\label{eq:defAn}
A_\rho=G_c^{-1}\left(\frac{1}{\rho} \right),
\end{equation} so that $I_\rho=\lfloor A_\rho+\frac{1}{2}\rfloor$. Since $G_c$ is a continuous strictly decreasing function, the term $A_\rho$ can be written as
  \[A_\rho=\inf\left\{x\in\RR_+ \mid G_c(x)\leq \frac{1}{\rho}\right\} =\sup\left\{x\in\RR_+ \mid G_c(x)\geq \frac{1}{\rho}\right\}.  \] It is enough to prove that $x_\rho-2\leq A_\rho\leq y_\rho$, where 
  \begin{equation*}
  x_\rho := \frac{d-1}2 \cdot \frac{\log \rho}{\log \log \rho} 
  \quad \text{and} \quad y_\rho := \frac{d-1}2 \left( \frac{\log \rho}{\log \log \rho}+2\,\frac{\log \rho}{(\log \log \rho)^2}\, \log \log \log \rho \right) .
  \end{equation*}
  
  To prove that $A_\rho\geq x_\rho-2$, we notice that
  \begin{equation*}
  A_\rho  \geq \inf\left\{k\in\NN \mid G(k-1)\leq \frac{1}{\rho}\right\} - 2.
   \end{equation*}
Besides, according to Proposition \ref{Prop:Gilles} and the fact that $G(k-1)=\PPP{\typ\geq k}$ is larger than $\PPP{\typ= k}$, we have $G(k-1)\geq c_1^k\, k^{-\frac{2}{d-1}k}$. Thus
\[ A_\rho  \geq \inf\left\{k\in\NN \mid c_1^kk^{\frac{-2}{d-1} k}\leq \frac{1}{\rho}\right\} - 2 \geq \inf\left\{x\in\RR_+ \mid c_1^{x}x^{\frac{-2}{d-1}x}\leq \frac{1}{\rho}\right\} - 2.\]
Moreover, for $\rho$ large enough, we have
  \[ c_1^{x_\rho} x_\rho^{\frac{-2}{d-1} x_\rho}
  = \exp \left( -\log \rho + \frac{\log \rho}{\log \log \rho}\, (\log \log \log \rho + O(1)) \right)
  > \frac{1}{\rho}. \]   In particular, we have $x_\rho\leq \inf\left\{x\in\RR_+ \mid c_1^{-x} x^{\frac{-2}{d-1}x}\leq \frac{1}{\rho}\right\}$, which proves that $A_\rho\geq x_\rho-2$.

  To prove that $A_\rho\leq y_\rho$, we proceed along the same lines as above. Indeed,
  \begin{equation*}
  A_\rho  \leq \sup\left\{k\in\NN \mid G(k-1)\geq \frac{1}{\rho}\right\}.
  \end{equation*}
Besides, because of Proposition \ref{Prop:Gilles}, there exists a constant $c_6>0$ such that, for each $k\in \NN$, we have  $G(k-1)\leq c_6^k\, k^{-\frac{2}{d-1}k}$. Thus
  \[   A_\rho \leq  \sup\left\{k\in\NN \mid c_6^k k^{\frac{-2}{d-1}k}\geq \frac{1}{\rho}\right\}  \leq \sup\left\{y\in\RR_+ \mid  c_6^y y^{\frac{-2}{d-1}y}\geq \frac{1}{\rho}\right\}.\] 
Moreover, with standard computations, we can easily show that
  \begin{equation*}
   c_6^{y_\rho}{y_\rho}^{\frac{-2}{d-1}y_\rho}    = \exp \left( -\log \rho - \frac{\log \rho}{\log \log \rho}\, (\log \log \log \rho + O(1))\right)
  < \frac{1}{\rho}.
  \end{equation*}
  In particular, we have   
  $y_\rho\geq \sup\left\{y\in\RR_+ \mid c_6^y y^{\frac{-2}{d-1}y}\geq \frac{1}{\rho}\right\} $,
  which proves that $A_\rho\leq y_\rho$.

  \subsection{Proof of Lemma \ref{Le:tail}}
  \textit{Proof of \eqref{Le:tail1}}.  First, we notice that for each $k\in\NN$, we have $\frac{G(k+1)}{G(k)}\conv[k]{\infty}0$. With standard computations (see e.g.\ the proof of Theorem 1 in \cite{And}), we easily show that 
  \begin{equation}
  \label{eq:rateG}
  \frac{G_c(x+y)}{G_c(x)}\conv[x]{\infty}0,
  \end{equation} for each $x,y\in\RR_+$. In particular, we get
  \[\rho\PPP{\typ\geq I_\rho+2} = \rho \, G_c(I_\rho+1) \leq \rho \, G_c\left(A_\rho+\frac{1}{2}\right) = o(\rho \, G_c(A_\rho)),\] where $A_\rho$ is defined in \eqref{eq:defAn}.  Since $G_c(A_\rho)=\frac{1}{\rho}$, we have $\rho\PPP{\typ\geq I_\rho+2} =o(1)$. 
  
 \noindent \textit{Proof of \eqref{Le:tail3}}. First, we deal with the case $l=0$. Proceeding in the same spirit as above, Equation \eqref{eq:rateG} gives
  \[ \rho\PPP{\typ \geq I_\rho} 
  = \rho \, G_c(I_\rho-1) 
  \geq \rho \, G_c\left(A_\rho-\frac{1}{2} \right)
  = \frac{G_c\left(A_\rho-\frac{1}{2} \right)}{G_c\left(A_\rho\right)}
  \conv[\rho]{\infty}\infty. \] 
  The general case follows from an induction on $l$ and from the following lines:
 \begin{align*}
 \rho \, \PPP{\typ \geq I_\rho-l} & \geq \rho\PPP{\typ = I_\rho-l}\\
 & \geq \rho\, c_2^{-1} (I_\rho-l+1)^{\frac{2}{d-1}}\PPP{\typ = I_\rho-l+1}\\
 & \eq[\rho]{\infty} c\, \rho\left(\frac{\log \rho}{\log\log \rho}\right)^{\frac{2}{d-1}}\PPP{\typ \geq  I_\rho-l+1},
 \end{align*}
  where the second inequality is a consequence of Proposition \ref{Prop:Gilles} and where the third line comes from  \eqref{eq:boundsGtypical} and the fact that $I_\rho \eq[\rho]{\infty}\frac{d-1}{2}\cdot \frac{\log \rho}{\log\log \rho}$.

\subsection{Proof of Proposition \ref{Prop:cluster5}}
 First, we show that if $N_\XX^B[k]\geq 5$, then almost surely there exists  at least one pair of nodes in $B$, with degree larger than $k$ but with few vertices in common. Then we show that such an event cannot occur with high probability. To do it, we begin with a result on deterministic geometric graphs, established in the following paragraph.

\paragraph{A bound for the number of common vertices in a deterministic geometric graph}

  \begin{Prop}
        \label{Prop:planargraph}
        Let $ G = ( V , E ) $ be a simple planar graph in $\RR^2$ and let $ S=\{s_1,s_2,s_3,s_4,s_5\}  \subset V$ be a set of five distinct vertices.
        Then there exist two vertices $ s_i, s_j \in S $ such that $ \#\left( \neighbours_G ( s_i ) \cap \neighbours_G ( s_j ) \setminus S \right) \leq 20 .$
    \end{Prop}
    To prove Proposition \ref{Prop:planargraph}, we will use the following lemma.
     \begin{Le}
        \label{lem:graph2}
        Let $ G = (V, E) $ be a simple planar graph in $\RR^2$ and let $S'=\{s_1, s_2, s_3\}\subset V$ be a set of three distinct vertices in $G$.  Then $ \#\left( \cap_{ i \leq 3 } \neighbours_G ( s_i ) \setminus S' \right) \leq 2 $.    \end{Le}
    \begin{proof}
   Assume, on the opposite, that there exists a set of three vertices, say $V'=\{v_1, v_2, v_3\}\subset V$, such that $v_j\in \cap_{ i \leq 3 } \neighbours_G ( s_i ) \setminus S' $ for each $1\leq j\leq 3$. The graph $G$ induces a complete bipartite graph with vertices $S' \sqcup V'$, which is the bipartite graph with partitions of size $(\#S', \#V') = (3,3)$, namely $K_{3,3}$. Since $K_{3,3}$ is not planar, this contradicts the fact that $G$ is planar. This concludes the proof of Lemma \ref{lem:graph2}. 
      \end{proof}
 \begin{proof}[Proof of Proposition \ref{Prop:planargraph}] Assume on the opposite that $ \#\left(\neighbours_G ( s_i ) \cap \neighbours_G ( s_j ) \setminus S \right) \geq 21 $ for each $1\leq i<j\leq 5$.  Let $ T \subset V \setminus S $ be the set of all vertices (not in $ S $) which are neighbors of at least three vertices in $S$.  According to Lemma~\ref{lem:graph2}, we know that $ \# T  \leq 2 \binom{ 5 }{ 3 } = 20$. In particular, for each $1\leq i < j\leq 5$, we have:
 \[ \#\left( (\neighbours_G ( s_i ) \cap \neighbours_G ( s_j ) \setminus S)\cap (V\setminus T)\right) \geq 1.\]
  Thus there exists a family of vertices $\{v_{ij}: 1\leq i<j\leq 5\}$ such that, for each $1\leq i<j\leq 5$, the vertex $v_{ij}$ belongs to $\neighbours_G ( s_i ) \cap \neighbours_G ( s_j )$ and $v_{ij}$ does not belong to $\neighbours_G ( s_k )$ for $k\neq i,j$. Notice that the vertices $v_{ij}$ are distinct and that all the edges $(s_i, v_{ij})$, $(s_j, v_{ij})$ are disjoint since $G$ is planar.
  
  Now we construct a graph $G'=(V',E')$ as follows. The set of vertices is $V'=S$. Any pair of vertices, say $\{s_i, s_j\}$, is connected  by an edge: this edge is defined as the union of $(s_i, v_{ij})$ and $(v_{ij}, s_j)$. On a one hand, the graph $G'$ is planar since $G$ is planar. But, on the other hand, $G'$ is the complete graph with five vertices, namely $K_5$, which is not planar. This gives a contradiction, which concludes the proof of Proposition  \ref{Prop:planargraph}.      
    \end{proof}

\paragraph{Proof of Proposition \ref{Prop:cluster5}}
Let $E_k$ be the following event:
\[E_k = \left\{ \sum_{x\in \XX\cap B}\ind{d_\XX(x)=k}\geq 5  \right\}.\]
Then
\begin{align*}
\PPP{N_\XX^B[k]\geq 5} & \leq \PPP{E_k}+\PPP{\bigcup_{x\in \XX\cap B}\{d_\XX(x)\geq k+1\}}\\
& \leq \PPP{E_k}+\EEE{\sum_{x\in \XX\cap B} \ind{d_\XX(x)\geq k+1}  }\\
& = \PPP{E_k}+ \vol{2}{B} \PPP{\typ\geq k+1}.
\end{align*}
  According to Proposition \ref{Prop:Gilles}, we know that
    \[\PPP{\typ\geq k+1}  \eq[\rho]{\infty}\PPP{\typ = k+1} \leq c \, k^{-2} \PPP{\typ=k}.\]
    
Now, we have to show that $\PPP{E_k} \leq c^k\, \vol{2}{B}^2\, k^{-2k/23} \PPP{\typ=k}$, which constitutes the main difficulty of the proof of Proposition \ref{Prop:cluster5}. To do it, we apply Proposition \ref{Prop:planargraph}: if the event $E_k$ occurs then there exist five nodes $x_1, \ldots, x_5$ in $\XX\cap B$ with degree $k$ such that  $ \#\left( \neighbours_{ \XX } ( x_1 ) \cap \neighbours_{ \XX } ( x_2 ) \right) \leq \#\left( \neighbours_{ \XX } ( x_1 ) \cap \neighbours_{ \XX } ( x_2 ) \setminus \{ x_1 , \ldots , x_5 \} \right) + 3 \leq 23 $. Thus
\[\PPP{E_k}\leq \PPP{\bigcup_{(x_1,x_2)\in \XX\cap B} \left\{ d_\XX(x_1) = d_\XX(x_2)=k  \right\} \cap \left\{ \#(\neighbours_\XX(x_1)\cap \neighbours_\XX(x_2))\leq 23    \right\}   }.\]
It follows from the multivariate Mecke equation and the fact that $\XX$ is stationary that
\[\PPP{E_k} \leq \vol{2}{B} \int_{2B}\PPP{  \left\{ d_{\XX\cup\{0,x\}}(0) = d_{\XX\cup\{0,x\}}(x) = k\right\}\cap \left\{   \#(\neighbours_{\XX\cup\{0,x\}}(0)\cap \neighbours_{\XX\cup\{0,x\}}(x))\leq 23       \right\}    }\mathrm{d}x.\]
Note that the integration domain is $2B$ since, because of the symmetry of $B$, this is precisely the set of all differences $x_2-x_1$ for $x_1,x_2\in B$.
We bound below the right-hand side. To do it, we introduce for any $l\in\NN$, $x\in \RR^2$, $s\in [0, \infty]$ the set $D_{l,x,s}\subset\RR^{2l}$ which consists of the family of $l$-tuples of points $q_{1:l}=(q_1,\ldots, q_l)$ in $\RR^2$  such that the following properties hold simultaneously:
\begin{figure}
\begin{center}
  \begin{tikzpicture}[line cap=round,line join=round,>=triangle 45,scale=0.6]
    \coordinate (o) at (-1.92,0.8);
    \coordinate (p1) at (1.5035128251299215,-0.9068564020092966);
    \coordinate (p2) at (-2.172341407380274,-3.303612709017635);
    \coordinate (p3) at (-6.556045948177905,-1.149015718002367);
    \coordinate (p4) at (-5.376295434365511,4.395811696915887);
    \coordinate (p5) at (-1.7252780547776823,5.097452791972732);
    \coordinate (p6) at (1.3793285605180905,3.532731057863662);
    \coordinate (p7) at (0.8080809433036683,0.949698353937577);
    \coordinate (x) at (4.58,0.7);
    \coordinate (q1) at (7.340173261885974,5.147126497817465);
    \coordinate (q2) at (10.407524597798197,0.3163586044172388);
    \coordinate (q3) at (8.259136820013524,-4.154274921608677);
    \coordinate (q4) at (2.0002498835772453,0.8193048760951543);
    \coordinate (q5) at (1.975413030654879,1.142183964085915);
  
    \def \r{0.5 ex}
    
    \draw[fill] (o) node[left] {$0$} circle (\r);
    \draw[fill] (p1) node[below] {$p_1$} circle (\r);
    \draw[fill] (p2) node[below] {$p_2$} circle (\r);
    \draw[fill] (p3) node[left] {$p_3$} circle (\r);
    \draw[fill] (p4) node[left] {$p_4$} circle (\r);
    \draw[fill] (p5) node[above] {$p_5$} circle (\r);
    \draw[fill] (p6) node[above] {$p_6$} circle (\r);
    \draw[fill] (p7) node[above left] {$p_7$} circle (\r);
    \draw[fill] (x) node[above right] {$x$} circle (\r);
    \draw[fill] (q1) node[above right] {$q_1$} circle (\r);
    \draw[fill] (q2) node[right] {$q_2$} circle (\r);
    \draw[fill] (q3) node[below right] {$q_3$} circle (\r);
    \draw[fill] (q4) node[below right] {$q_4$} circle (\r);
    \draw[fill] (q5) node[above right] {$q_5$} circle (\r); 
    
    \foreach \i in {1,...,7} {\draw (o) -- (p\i);}
    \draw (p1) -- (p2) -- (p3) -- (p4) -- (p5) -- (p6) -- (p7) -- (p1);
    \foreach \i in {1,...,5} {\draw (x) -- (q\i);}
    \draw (p1) -- (x) -- (p6);
    \draw (q1) -- (q2) -- (q3) -- (p1) -- (q4) -- (q5) -- (p6) -- (q1);
    \draw (q5) -- (p7) -- (q4);
  
    \tkzCircumCenter(x,p1,q3)\tkzGetPoint{o1} \tkzDrawArc[color=black,very thick](o1,p1)(q3)
    \tkzCircumCenter(x,q4,p1)\tkzGetPoint{o2} \tkzDrawArc[color=black,very thick](o2,q4)(p1)
    \tkzCircumCenter(x,q5,q4)\tkzGetPoint{o3} \tkzDrawArc[color=black,very thick](o3,q5)(q4)
    \tkzCircumCenter(x,p6,q5)\tkzGetPoint{o4} \tkzDrawArc[color=black,very thick](o4,p6)(q5)
    \tkzCircumCenter(x,q1,p6)\tkzGetPoint{o5} \tkzDrawArc[color=black,very thick](o5,q1)(p6)
    \tkzCircumCenter(x,q2,q1)\tkzGetPoint{o6} \tkzDrawArc[color=black,very thick](o6,q2)(q1)
    \tkzCircumCenter(x,q3,q2)\tkzGetPoint{o7} \tkzDrawArc[color=black,very thick](o7,q3)(q2)
    
    \tkzDrawArc[color=black,fill,opacity=0.3, very thin](o6,q2)(x)
    \tkzDrawArc[color=black,fill,opacity=0.3, very thin](o7,x)(q2)
  \end{tikzpicture}
\end{center}
\caption{\label{fig:Dlxs} Part of the Delaunay graph $\del(\XX\cup\{0,x\})$ around the points $0$ and $x$.}
Black thick curve: contour of the Voronoi flower centered at $x$.\\
Gray region: union of the circumscribed disks of $\{x,q_1,q_2\}$ and $\{x,q_2,q_3\}$.
\end{figure}
\[ \mathscr{P}:  \left\{ \begin{split} &   q_1, \ldots, q_{l} \text{ are clockwise ordered around $x$ };\\
& q_j\not\in B(x,q_i,q_{i+1}) \text{ for any $i< l$ and $j\leq l$};\\
& \vol{2}{ \bigcup_{i< l}B(x,q_i,q_{i+1}) } \leq s.
 \end{split}     \right.\]
Here ``clockwise ordered around $x$'' means that the points appear in order when viewed from $x$ and turning clockwise.
These properties are illustrated by Figure \ref{fig:Dlxs}.
In this figure the points $q_1$, $q_2$ and $q_3$ are three consecutive neighbors of $x$ (clockwise ordered around $x$). The circumscribed disks of $\{ x , q_1 , q_2 \}$ and $\{ x , q_2 , q_3 \}$ are petals of the Voronoi flower centered at $x$, and therefore the area of their union is less then $\Phi_\XX(x)$. These facts imply that $q_{1:3}$ is an element of $D_{x,3,\Phi_\XX(x)}$.

Note that, contrary to the set $C_k$ introduced just before Lemma \ref{Le:explicitqk}, the set $D_{l,x,s}$  is not stable under coordinates permutation.
This is due to the clockwise orientation restriction.
We will also use several times the following homogeneity properties which hold for any $l\in\NN$, $x\in\RR^2$ and $0<s<t$, 
\begin{equation} \label{eq:homogeneity}
    D_{l,x,s} = x + D_{l,0,s}, \qquad
    D_{l,x,s} \subset D_{l,x,t}, \qquad 
    V_{2l}(D_{l,x,s}) = s^l V_{2l}(D_{l,x,1}).
\end{equation}

Now, let $x\in B$ be fixed. Assume that the following events $\{d_{\XX\cup\{0,x\}}(0) = d_{\XX\cup\{0,x\}}(x) = k\}$, $\left\{   \#(\neighbours_{\XX\cup\{0,x\}}(0)\cap \neighbours_{\XX\cup\{0,x\}}(x))\leq 23       \right\}$ and $ \{\vol{2}{F_{\XX\cup\{0,x\}}(x)}\leq \vol{2}{F_{\XX\cup\{0,x\}}(0)}\}$ hold simultaneously. In particular, there exist at least $k-23$ neighbors of $x$ which do not belong to the Voronoi flower $F_{\XX\cup\{0,x\}}(0)$. Thus there exists at least $k'=\left\lfloor \frac{k-23}{23}\right\rfloor$ \textit{consecutive} (clockwise ordered around  $x$) neighbors of $x$, which are not neighbors of 0. Thus  there exists a $k'$-tuple of points $p_{1:k'}\in\XX^{k'}$ such that the family of properties $\mathscr{P}$ holds, with $l=k'$ and $s=\vol{2}{F_{\XX\cup\{0,x\}}(0)}$. Therefore \[\PPP{E_k}\leq 2\, V_2(B) \int_{2B}\PPP{\{d_{\XX\cup\{0,x\}}(0)=k\} \cap \{(\XX\setminus \neighbours_{\XX\cup\{0,x\}}(0))^{k'}_{\neq}\cap D_{k',x, \vol{2}{ F_{\XX\cup\{0,x\}}(0) } }\neq \varnothing\}   }\mathrm{d}x.\] 
The factor 2 comes from the fact that $ \vol{2}{F_{\XX\cup\{0,x\}}(x)}$ was assumed to be less than $\vol{2}{F_{\XX\cup\{0,x\}}(0)}$. Now, we discuss two cases: the first one is when $x$ and 0 are not neighbors and the second one deals with the complement event.

\paragraph{Case 1. The nodes $x$ and 0 are not neighbors}
In this case, we bound for any $x\in 2 B$, the following probability:
\[P_1(x) = \PPP{\{d_{\XX\cup\{0,x\}}(0)=k\} \cap \{(\XX\setminus \neighbours_{\XX\cup\{0,x\}}(0))^{k'}_{\neq}\cap D_{k',x, \vol{2}{F_{\XX\cup\{0,x\}}(0)}}\neq \varnothing\}  \cap \{x\not\in \neighbours_{\XX\cup\{0,x\}}(0)\}}.\]
To do it, we write
\begin{equation*}
P_1(x) \leq  \frac{1}{k!} \EEE{ \sum_{(p_{1:k},q_{1:k'})\in\XX^{k+k'}_{\neq}} \ind{F_{p_{1:k}}(0)\cap \XX=\varnothing  } \ind{p_{1:k}\in C_k}   \ind{q_{1:k'}\in D_{k',x, \Phi_{p_{1:k}}(0)}}  %\ind{x\not\in \neighbours_{\XX\cup\{0,x\}}(0)}
},
\end{equation*}
where we recall that $F_{p_{1:k}}(0)$ is the Voronoi flower with nucleus $0$ induced by the set of points $\{0,p_1, \ldots, p_k\}$. Notice that we have divided by $k!$ because $C_k$ is stable under permutations which is not the case for $D_{l,x,s}$.
% Bounding the last indicator function by 1, i
It follows from the multivariate Mecke equation that 
\begin{equation}
\label{eq:case1notneighbor}
P_1(x)  \leq  \frac{1}{k!} \int_{\RR^{2k}}\int_{\RR^{2k'} } \PPP{F_{p_{1:k}}(0) \cap \XX=\varnothing   }
\ind{p_{1:k}\in C_k}   \ind{q_{1:k'}\in  D_{k',x,\Phi_{p_{1:k}}(0)}  }\mathrm{d}q_{1:k'}\mathrm{d}p_{1:k}.
\end{equation}
Integrating over $q_{1:k'}$, it follows from Fubini's theorem and the fact that $\XX$ is a Poisson point process, that
\[
P_1(x)  \leq \frac{1}{k!} \int_{C_k}e^{-\Phi_{p_{1:k}}(0)}\, \vol{2k'}{D_{k',x,\Phi{p_{1:k}}(0)}} \mathrm{d}p_{1:k}.
\]
As in the proof of Lemma \ref{Le:explicitqk}, we use the fact that $e^{-\Phi_{p_{1:k}}(0)}=\int_0^\infty e^{-s}\ind{\Phi_{p_{1:k}}(0)\leq s}\mathrm{d}s$.
% Besides, we use the fact that $D_{k',x,\Phi_{p_{1:k}}(0)}$ is included in $D_{k',x,s}$ whenever the event $\{\Phi_{p_{1:k}}(0)\leq s\}$ is satisfied. Taking the change of variables $p_{1:k} = s^{1/2}y_{1:k}$, this gives:
The change of variables $p_{1:k} = s^{1/2}y_{1:k}$ and the properties \eqref{eq:homogeneity} give
\[P_1(x)\leq  \frac{1}{k!} \int_0^\infty e^{-s}\int_{C_k}\ind{\Phi_{y_{1:k}}(0)\leq 1}s^{k+k'}\, \vol{2k'}{D_{k',x, 1}} \mathrm{d}y_{1:k}\mathrm{d}s.\]
Integrating over $s$, we deduce from Lemma \ref{Le:explicitqk} that 
\[P_1(x) \leq \frac{(k+k')!}{k!} \PPP{\typ=k}\, \vol{2k'}{D_{k',x, 1}}.\]
The next lemma provides an upper bound for $\vol{2k'}{D_{k',x, 1}}$. 
\begin{Le}
\label{Le:majvol}
There exists a constant $c>0$ such that, for any $ j \in \NN $ and $ x \in \RR^2 $,
\[ \vol{2j}{D_{j,x, 1}} \leq  \frac{c}{(j-1)!}\PPP{\typ=j } . \]
\end{Le}
\begin{proof}
  First, we notice that this term actually does not depend on $x$ since $D_{j,x,s}=x+D_{j,0,s}$ for any $j,x,s$. 
  In the proof of this lemma we will use the notation $ \tilde{\Phi}_{p_{1:j}}( x ) := \vol{2}{ \cup_{ i \leq j - 1 } B ( x , p_i , p_{ i + 1 } ) } $ for any $ p_{1:j} \in D_{j,x,\infty} $. 
  Similarly as above, we combine the substitution $ p_{1:j} = s^{1/2} y_{1:j} $ with the observation that $ \int_0^\infty e^{-s}\ind{\tilde{\Phi}_{p_{1:j}}(0)\leq s}\mathrm{d}s = e^{-\tilde{\Phi}_{p_{1:j}}(0)}$.
  This gives
  \begin{align*}
    \vol{2j}{D_{j,0, 1}} 
    %& =  \int_{ D_{ I_\rho - 23 , 0 , \infty } } \ind{ \left| \cup_{ i \in [ l_\rho - 24 ] } B ( 0 , y_i , y_{ i + 1 } ) \right| \leq 1 } \, \dint y \\
    & = \frac{ 1 }{ j ! } \int_0^\infty e^{ - s } \int_{ D_{ j , 0 , \infty } } \ind{ \tilde{\Phi}_{y_{1:j}}(0) \leq 1 }  s^{ j } \, \dint y \, \dint s %\\
    % & = \frac{ 1 }{ j ! } \int_0^\infty e^{ - s } \int_{ D_{ j , 0 , \infty } } \ind{ \tilde{\Phi} (p_{1:j}) \leq s }  \, \dint p \, \dint s \\ &
    = \frac{ 1 }{ j ! } \int_{ D_{ j , 0 , \infty } } e^{-\tilde{\Phi}_{p_{1:j}}(0)} \, \dint p .
  \end{align*}
  Using the fact that $ e^{-\tilde{\Phi}_{p_{1:j}}(0)} = \EEE{ \ind{ \XX \cap \cup_{ i \leq j-1 } B(0,p_i,p_{i+1}) = \emptyset } } $, the multivariate Mecke equation implies that 
  \begin{align*}
    \vol{2j}{D_{j,0, 1}}  
    & = \frac{ 1 }{ j ! } \EEE{ \sum_{ p_{1:j} \in \XX_{\neq}^{j} } \ind{ p_{1:j} \in D_{ j , 0 , \infty } } \ind{ \XX \cap \cup_{ i \leq j-1 } B(0,p_i,p_{i+1}) = \emptyset } } .
  \end{align*} 
  Note that {a.s. the random variable $d_{\XX\cup\{0\}}(0)$ is larger than $j$ }  whenever the indicator functions above are equal to one.
  Moreover if ${d_{\XX\cup\{0\}}(0)} = l \geq j$, then there exist exactly $ l $ tuples of points $p_{1:l}$ in $\XX$ such that the corresponding events hold.
  In fact $p_1$ must be a neighbor of $ 0$ in $Del(\XX) $, and picking it arbitrarily implies that $ p_2 , p_3 , \ldots $ are the (clockwise oredered around $0$) neighbors of $0$. 
  Thus, {according to \eqref{eq:proptypical}}, we can write
  \begin{align*}
      \vol{2j}{D_{j,0, 1}} 
      & = \frac{ 1 }{ j ! } \EEE{ \sum_{ l \geq j } l \, \ind{ \typ = l } } 
      = \frac{ 1 }{ j ! } \sum_{ l \geq j } l \PPP{ \typ = l }  .
  \end{align*} 
  We conclude the proof by using the estimates of Proposition \ref{Prop:Gilles}.
\end{proof}

  According to Lemma \ref{Le:majvol}, for any $x\in \RR^2$, we have
\[P_1(x)  \leq c\, \frac{(k+k')!}{k!(k'-1)!} \PPP{\typ=k} \PPP{\typ=k'} \leq c^{k}\, k^{-2k/23} \PPP{\typ=k},\] where the second inequality is a consequence of Proposition \ref{Prop:Gilles} and the fact that $k'=\left\lfloor \frac{k-23}{23}\right\rfloor$.  Integrating over $x\in 2 B$, we deduce that $\int_{2 B}P_1(x)\mathrm{d}x \leq c^{k}\, \vol{2}{B}\,  k^{-2k/23} \PPP{\typ=k}$.

\paragraph{Case 2. The nodes $x$ and 0 are neighbors}
In this case,  for any $x\in B$, we deal with the following probability:
\[P_2(x) = \PPP{\{d_{\XX\cup\{0,x\}}(0)=k\} \cap (\XX\setminus \neighbours_{\XX\cup\{0,x\}}(0))^{k'}_{\neq} \cap D_{k',x, \vol{2}{F_{\XX\cup\{0,x\}}(0)}}\neq \varnothing\}  \cap \{x\in \neighbours_{\XX\cup\{0,x\}}(0)\}}.\]
Since we now consider situations where $x$ is one of the $k$ neighbors of $0$, it will be practical in the following lines to set $p_k = x$ in order to keep relatively short notation.
This time we write 
% \begin{multline*}
% P_2(x) \leq \frac{1}{(k-1)!} \EE\Big[\sum_{p_{1:k-1}\in \XX^{k-1}}\sum_{q_{1:k-23}\in \XX^{k-23}}    \ind{F_{\{0,x,p_1,\ldots,p_{k-1}\}}(0)\cap \XX=\varnothing  } \\
% \times  \ind{(x,p_1,\ldots,p_k)\in C_k}   \ind{q_{1:k-23}\in D_{k-23,x, \Phi_{\{x,p_1,\ldots,p_{k-1}\}}(0)}} % \ind{x\in F_{\{0,x,p_1,\ldots,p_{k-1}\}}(0)}
% \Big].
% \end{multline*}
\begin{equation*}
P_2(x) = P_2(p_k) \leq \frac{1}{(k-1)!} \EEE{\sum_{(p_{1:k-1},q_{1:k'})\in\XX^{k-1+k'}_{\neq}}  \ind{F_{p_{1:k}}(0)\cap \XX=\varnothing  } 
\ind{p_{1:k}\in C_k}   \ind{q_{1:k'}\in D_{k',p_k, \Phi_{p_{1:k}}(0)}} }.
\end{equation*}
 Integrating over $x\in B$ and applying the multivariate Mecke equation as in the first case, we have
%      \begin{multline*}
% \int_{B}P_2(x)\mathrm{d}x  \leq  \frac{1}{(k-1)!} \int_{B\times\RR^{2(k-1)}}\int_{\RR^{2(k-23)} } \PPP{F_{  \{0, x,p_1,\ldots,p_{k-1}\}}(0) \cap \XX=\varnothing   }\\
% \times \ind{\{x,p_1,\ldots,p_{k-1}\}\in C_k}   \ind{q_{1:k-23}\in  D_{k-23,x,\Phi(0;x,p_{1:k-1})}  }\mathrm{d}q_{1:k-23}\mathrm{d}p_{1:k-1}\mathrm{d}x.
% \end{multline*}       
     \begin{equation*}
\int_{B}P_2(x)\mathrm{d}x  \leq  \frac{1}{(k-1)!} \int_{\RR^{2(k-1)}\times B}\int_{\RR^{2k'} } \PPP{F_{p_{1:k}}(0) \cap \XX=\varnothing   } \ind{p_{1:k}\in C_k}   \ind{q_{1:k'}\in  D_{k',p_k,\Phi_{p_{1:k}}(0)}  }\mathrm{d}q_{1:k'}\mathrm{d}p_{1:k}.
\end{equation*}       
The right-hand side is very similar to the upper bound in \eqref{eq:case1notneighbor}. There are only two differences between these upper bounds. The first one is that we integrate over $\RR^{2(k-1)}\times B$ instead of $\RR^{2k}$. The second one is that we consider the ratio $\frac{1}{(k-1)!}$ instead of $\frac{1}{k!}$. However, proceeding exactly along the same lines as in the first case, we obtain that  $\int_{2 B}P_2(x)\mathrm{d}x \leq c^{k}\, k^{-2k/23} \PPP{\typ=k}$. 

Since $\PPP{E_k}\leq 2 \vol{2}{B} \int_{ 2 B}(P_1(x)+P_2(x))\mathrm{d}x$, it follows from the two cases discussed above that $\PPP{E_k}\leq c^{k}\, \vol{2}{B}^2\,  k^{-2k/23} \PPP{\typ=k} $. This concludes the proof of Proposition \ref{Prop:cluster5}.

\subsection{Proof of Proposition \ref{Prop:maxblock}}
Recall that $M^B_\XX = \max_{x\in \XX\cap B} d_{\XX}(x)$ and $N^B_{\XX}[k] = \sum_{x\in \XX\cap B}\ind{d_{\XX}(x)\geq k}$ denote the maximum degree and the number of exceedances in the set $B$  respectively. This gives
\begin{align*}  \PPP{\typ\geq k} &  =\frac{1}{\vol{2}{B}}\EEE{\sum_{x\in \XX\cap B}\ind{d_{\XX}(x)\geq k}}\\
& = \frac{1}{\vol{2}{B}}\EEE{ N^B_{\XX}[k]\ind{N^B_{\XX}[k]\leq 4}\ind{M^B_{\XX}\geq k}} +  \frac{1}{\vol{2}{B}} \EEE{N^B_{\XX}[k]\ind{N^B_{\XX}[k]\geq 5}}.
\end{align*}
We bound $N^B_{\XX}[k] \ind{N^B_{\XX}[k]\leq 4}$ by $4$ in the first expectation and $N^B_{\XX}[k]$ by $\#(\XX\cap B)$ in the second one. We get
\begin{equation}
\label{eq:lowerboundtyp}
\PPP{\typ\geq k}  \leq \frac{4}{\vol{2}{B}} \PPP{M^B_{\XX}\geq k} + \frac{1}{\vol{2}{B}} \EEE{\#(\XX\cap B)\ind{N^B_{\XX}[k]\geq 5}}
\end{equation}
We show below that the second term of the right-hand side equals $o\left(\PPP{\typ=k}\right)$. To do it, we write
\begin{equation*}  \EEE{\#(\XX\cap B)\ind{N^B_{\XX}[k]\geq 5}}
 \leq k^{1+\epsilon}\vol{2}{B}\PPP{N_\XX^B[k]\geq 5} + \EEE{\#(\XX\cap B)\ind{\#(\XX\cap B)\geq k^{1+\epsilon} \vol{2}{B} }},
\end{equation*}
for some $\epsilon\in (0,1)$. According to Proposition \ref{Prop:cluster5}, since $\epsilon<1$, we have \[k^{1+\epsilon}\vol{2}{B}\PPP{N^\XX_B[k]\geq 5} = o\left(\PPP{\typ =k} \right)\] as $k$ goes to infinity.  Moreover, since $\#\,\XX\cap B$ is a Poisson random variable with parameter $\vol{2}{B}$, it follows from standard computations that
\begin{align*}
\EEE{\#(\XX\cap B)\ind{\#(\XX\cap B)\geq k^{1+\epsilon}\vol{2}{B}}} & \eq[k]{\infty} k^{1+\epsilon} \vol{2}{B}\PPP{\#\,\XX\cap B \geq   k^{1+\epsilon}\vol{2}{B}}\\
& \leq  k^{1+\epsilon} \vol{2}{B}\,\exp\left(- k^{1+\epsilon}\vol{2}{B}\right)\EEE{e^{\#\eta\cap B}},
\end{align*}
where the second line is a consequence of the Markov's inequality. Besides, according to Proposition \ref{Prop:Gilles}, we have  $k^{1+\epsilon} \exp\left(- k^{1+\epsilon}\vol{2}{B}\right) = o\left( \PPP{\typ=k}  \right)$. This implies that \[\EEE{\#(\XX\cap B)\ind{N^B_{\XX}[k]\geq 5}}  = o\left( \PPP{\typ=k}  \right).\] This together with \eqref{eq:lowerboundtyp} concludes the proof of Proposition \ref{Prop:maxblock}.

\subsection{Proof of Proposition \ref{Prop:maxblockd}}
Proceeding along the same lines as in the proof of Proposition \ref{Prop:maxblock} (see Equation \eqref{eq:lowerboundtyp}), we obtain for any $k\in \NN$, $h\geq 1$  that
\begin{equation*}
\PPP{\typ\geq k}  \leq \frac{h+1}{\vol{d}{B}} \PPP{M^B_{\XX}\geq k} + \frac{1}{\vol{d}{B}} \EEE{N^B_{\XX}[k]\ind{N^B_{\XX}[k]\geq h+1}}.
\end{equation*}
To deal with the second term of the right-hand side, we apply the multivariate Mecke equation. This gives
\begin{align*}
\EEE{N^B_{\XX}[k]\ind{N^B_{\XX}[k]\geq h+1}} & = \EEE{\sum_{x\in \XX\cap B}\ind{d_\XX(x)\geq k}\ind{N^B_{\XX}[k]\geq h+1}}\\
& = \int_B\PPP{d_{\XX\cup\{x\}}(x)\geq k,N^B_{\XX\cup\{x\}}[k]\geq h+1}\mathrm{d}x.
\end{align*}
Bounding the integrand by the probability of the event $\{\#((\XX\cup\{x\})\cap B)\geq h+1\}$, which equals $\{\#\,\XX\cap B\geq h \}$ for almost all $x\in B$, we obtain
\begin{align*}
\EEE{N^B_{\XX}[k]\ind{N^B_{\XX}[k]\geq h+1}} & \leq \vol{d}{B} \PPP{\#\,\XX\cap B\geq h}.
\end{align*}
Thus
\begin{equation} \label{eq:lowerboundtypd}\PPP{\typ\geq  k} \leq \frac{h}{\vol{d}{B}} \PPP{M^B_{\XX}\geq k} + \PPP{\#\,\XX\cap B\geq h}.
\end{equation}
Since the random variable $\#\,\XX\cap B$ is Poisson distributed with parameter $\vol{d}{B}$, it follows from the Markov's inequality that 
\[\PPP{\#\,\XX\cap B\geq h} \leq e^{-h}\EEE{e^{\#\,\XX\cap B}} = \exp\left( -\vol{d}{B}\left(  1-e+\frac{h}{\vol{d}{B}} \right)  \right).\]  This together with \eqref{eq:lowerboundtypd} concludes the proof of Proposition \ref{Prop:maxblockd}.

\subsection{Proof of Lemma \ref{Le:union}}

For any $\omega\in \Omega$, let $d(\omega)$ be the number of $B^{(i)}$'s which contain $\omega$. Let also $a_i(l)=\PPP{\{\omega\in B^{(i)}: d(\omega)=l\}}$ for any $1\leq i, l\leq K$. According to Lemma 1 in \cite{KAT}, we know that
\[\PPP{  \bigcup_{ i =1}^K B^{(i)}  } = \sum_{i=1}^{K}\sum_{l=1}^K\frac{a_i(l)}{l}.\] Moreover, according to the assumption, we have $a_i(l)=0$ for any $i\leq K$ and $l>k$. Thus 
 \[\PPP{  \bigcup_{ i =1}^K B^{(i)}  } \geq \frac{1}{k}  \sum_{i=1}^{K}\sum_{l=1}^{k}a_i(l) = \frac{1}{k}\sum_{i=1}^K\PPP{B^{(i)}}.\]

\end{document}